\documentclass[11pt]{amsart}
\usepackage{etex}
\usepackage{pst-all}
\usepackage{amsmath}
\usepackage{amsfonts}
\usepackage{amssymb}
\usepackage{multicol}
\usepackage{verbatim}
\usepackage{amsthm}
\usepackage{amscd}
\usepackage{pb-diagram}
\usepackage[mathscr]{eucal}
\usepackage[all]{xy}
\usepackage{tikz}

\usepackage{amsmath}

\usepackage{amsfonts}

\usepackage{amssymb}

\usepackage{amsthm}

\usepackage{graphics}

\newcommand{\ZZ}{\mathbb{Z}}

\newcommand{\CC}{\mathbb{C}}

\newcommand{\NN}{\mathbb{N}}

\newcommand{\wt}{\mathrm{wt}}

\newcommand{\ffbox}[1]{
\setbox9=\hbox{$\scriptstyle\overline{1}$}
\framebox[20pt][c]{\rule{0mm}{\ht9}${\scriptstyle #1}$}
}

\newcommand{\QQ}{\mathbb{Q}}

\newcommand{\Hlie}{\mathfrak{h}}

\newcommand{\U}{\mathcal{U}}

\newtheorem{thm}{Theorem}[section]

\newtheorem{defi}[thm]{Definition}

\newtheorem{prop}[thm]{Proposition}

\newtheorem{lem}[thm]{Lemma}

\newtheorem{conj}[thm]{Conjecture}

\newtheorem*{conjecture}{Conjecture}

\newtheorem{rem}[thm]{Remark}

\newtheorem{ex}[thm]{Example}

\usepackage[all]{xy}
\title{Extremal loop weight modules for $\U_q(\hat{sl}_\infty)$}

\author[Mathieu Mansuy]{Mathieu Mansuy}
\address{Univ. Paris-Diderot-Paris 7, IMJ - PRG CNRS UMR 7586, B\^at. Sophie Germain, Case 7012, 75205 Paris Cedex 13, FRANCE}
\email{mansuy@math.jussieu.fr}

\begin{document}

\begin{abstract}
We construct by fusion product new irreducible representations of the quantum affinization $\U_q(\hat{sl}_\infty)$. The action is defined via the Drinfeld coproduct and is related to the crystal structure of semi-standard tableaux of type $A_\infty$. We call these representations extremal loop weight modules. The main motivations are applications to quantum toroidal algebras $\U_q(sl_{n+1}^{tor})$: we prove the conjectural link between $\U_q(\hat{sl}_\infty)$ and $\U_q(sl_{n+1}^{tor})$ stated in \cite{hernandez_algebra_2011} for this family of representations. We recover in this way the extremal loop weight modules obtained in \cite{mansuy_extremal_2013}.
\end{abstract}

\maketitle

\tableofcontents

\section{Introduction}

Let $\U_q(sl_\infty)$ be the quantum group associated to the infinite Dynkin diagram

\bigskip

\begin{center} \begin{tikzpicture}
\tikzstyle{point}=[circle,draw]
\tikzstyle{ligne}=[thick]
\tikzstyle{pointille}=[thick,dotted]

\node (1) at (-4, 0) []{};
\node (2) at ( -2,0) [point] {};
\node (3) at ( -1,0) [point] {};
\node (4) at ( 0,0) [point] {};
\node (5) at ( 1,0) [point] {};
\node (6) at (2,0) [point] {};
\node (7) at (4,0) [] {};

\draw [pointille] (1) -- (2);
\draw [ligne] (2) -- (3);
\draw [ligne] (3) -- (4);
\draw [ligne] (4) -- (5);
\draw [ligne] (5) -- (6);
\draw [pointille] (6) -- (7);

\end{tikzpicture}\end{center}

\bigskip

It can be considered as limit of the quantum groups of type $A_n$ and has been studied by various authors (see for example \cite{ariki_factorization_2012, enomoto_symmetric_2008, frenkel_hopf_2002, levendorskii_quantum_1991} and references therein). In this paper we study the quantum affinization $\U_q(\hat{sl}_\infty)$ of $\U_q(sl_\infty)$ \cite{hernandez_representations_2005, nakajima_quiver_2001}. This algebra is introduced in \cite{hernandez_algebra_2011} and can be viewed as the limit of the quantum affine algebras $\U_q(\hat{sl}_{n+1})$ when $n \rightarrow \infty$.

The quantum affinizations, in particular the quantum affine algebras and the quantum toroidal algebras, have been intensively studied (see for example \cite{chari_quantum_1991, feigin_representations_2013, frenkel_$q$-characters_2002, frenkel_$q$-characters_1999, hernandez_representations_2005, hernandez_quantum_2009, hernandez_algebra_2011, nakajima_quiver_2001} and references therein). In recent works \cite{mansuy_quantum_2012, mansuy_extremal_2013} we constructed new families of integrable representations of the quantum toroidal algebra $\U_q(sl_{n+1}^{tor})$, called extremal loop weight modules, which generalize the $\ell$-highest weight modules: there are representations generated by an extremal vector for the horizontal quantum affine subalgebra in the sense of Kashiwara \cite{kashiwara_crystal_1994, kashiwara_level-zero_2002}. The main motivation is the construction of finite-dimensional representations of the quantum toroidal algebra at roots of unity.

The representation theory of the algebra $\U_q(\hat{sl}_\infty)$ is related to the one of quantum toroidal algebras $\U_q(sl_{n+1}^{tor})$ in the following way (see \cite{hernandez_algebra_2011}) : let us consider the morphism between the corresponding Dynkin diagrams of the algebras $\U_q(\hat{sl}_\infty)$ and $\U_q(sl_{n+1}^{tor})$
$$\phi_n : i \in \ZZ \mapsto \overline{i} \in \ZZ / (n+1) \ZZ.$$
It gives rise to a ring morphism
$$\phi_n : \ZZ[Y_{i,a}^{\pm 1}]_{i \in \ZZ, a \in \CC^{\ast}} \longrightarrow \ZZ[Y_{i,a}^{\pm 1}]_{i \in I_n, a \in \CC^{\ast}}.$$
Then the following combinatorial link between the algebras $\U_q(\hat{sl}_\infty)$ and $\U_q(sl_{n+1}^{tor})$ is expected in \cite{hernandez_algebra_2011}.

\begin{conjecture}\cite[Conjecture 5.3]{hernandez_algebra_2011}
Let $V$ be a simple $\U_q(\hat{sl}_\infty)$-module in the category $\mathcal{O}_{\mathrm{int}}$. Then the image of the $q$--character of $V$ by $\phi_n$ is also the $q$--character of a representation of $\U_q(sl_{n+1}^{tor})$.
\end{conjecture}

\noindent The main motivation in \cite{hernandez_algebra_2011} is to predict $q$--character formulae for representations of $\U_q(sl_{n+1}^{tor})$. This conjecture is proved for the class of Kirillov-Reshetikhin modules of $\U_q(sl_{n+1}^{tor})$.

\medskip

The aim of this article is twofold. First construct extremal loop weight modules for $\U_q(\hat{sl}_\infty)$: we obtain here a large class of such modules with basis labelled by semi-standard tableaux. Second by using the combinatorial link with the quantum toroidal algebras, construct extremal loop weight modules for $\U_q(sl_{n+1}^{tor})$: we prove the conjecture above for the particular family of extremal fundamental loop weight modules. We recover in this way the extremal loop weight modules defined in \cite{mansuy_extremal_2013} for the quantum toroidal algebras.

\medskip

Let us explain this in more details. The construction of extremal loop weight modules is inspired by the original Kashiwara's study of the extremal weight modules (see \cite{kashiwara_crystal_1994}): we consider the fusion product of fundamental $\ell$-highest weight modules and fundamental $\ell$-lowest weight modules associated respectively to the fundamental weights  $\Lambda_\ell$ and $-\Lambda_0$ ($\ell \geq 1$) and to a non-zero complex parameter. The action of $\U_q(\hat{sl}_\infty)$ is defined by the Drinfeld coproduct, in the spirit of \cite{mansuy_extremal_2013}. We show that we get in this way extremal loop weight modules associated to the weight $\Lambda_\ell - \Lambda_0$, called extremal fundamental loop weight modules. By fusion product of $k$ extremal fundamental loop weight modules, we obtain extremal loop weight modules associated to the weight $k \Lambda_\ell - k \Lambda_0$ with basis labelled by the set of semi-standard tableaux $\mathcal{T}_{[1, \ell] \times [1, k]}$ of shape $(\ell \times k)$. Furthermore, the action of $\U_q(\hat{sl}_\infty)$ is explicitly describe on all these modules and involves the $\U_q(sl_\infty)$-crystal structure on $\mathcal{T}_{[1, \ell] \times [1, k]}$.

Our main motivation are applications to quantum toroidal algebras. We prove the conjecture above for the family of extremal fundamental loop weight modules of $\U_q(\hat{sl}_\infty)$ of weight $\Lambda_\ell - \Lambda_0$ introduced in this paper: we determine $q$--character formulae for these representations. Furthermore, we show that their image by the morphism $\phi_n$ is in fact the $q$--character of a representation of $\U_q(sl_{n+1}^{tor})$ constructed in \cite{mansuy_extremal_2013}. Recall that our original goal is to construct extremal loop weight modules. We show that we get in this way extremal loop weight modules if and only if
\begin{align*}
(\ell = 1) \text{ or } (n=2r+1 \text{ and } \ell=r+1).
\end{align*}

Let us explain another motivation to consider the algebra $\U_q(\hat{sl}_\infty)$. Recall that it is defined as limit of the quantum affine algebras $\U_q(\hat{sl}_{n+1})$ $(n \geq 2)$. For this reason the representation theory of this algebra is better understood than for general quantum affinizations: the fundamental modules we have to consider are thin, with basis labelled by semi-standard tableaux on which the action is explicitly known. It does not hold in the quantum toroidal case: in fact for $\U_q(sl_{n+1}^{tor})$ the fundamental modules are not thin (see \cite[Section 4.1]{hernandez_algebra_2011}). This is one of the main reason of the met difficulties in the study of extremal loop weight modules for the quantum toroidal algebras (besides see \cite[Remark 3.9]{mansuy_extremal_2013}). We will see that more precise results can be obtained in our construction of extremal loop weight modules for the quantum affinization $\U_q(\hat{sl}_\infty)$.

\medskip

The paper is organized as follows.

In Section 2 we recall the main definitions of the quantum group $\U_q(sl_\infty)$ and its quantum affinization $\U_q(\hat{sl}_\infty)$ and we briefly review their representation theory. In particular one defines the extremal weight modules and we introduce the notion of extremal loop weight modules for $\U_q(\hat{sl}_\infty)$. In Section 3 we construct the fusion product of fundamental $\ell$-highest weight modules and fundamental $\ell$-lowest weight modules associated respectively to the weights $\Lambda_\ell$ and $-\Lambda_0$ ($\ell \geq 1$) and to a non-zero complex parameter. We define a $\U_q(\hat{sl}_\infty)$-module structure on it by using the Drinfeld coproduct (Proposition \ref{tpexistudinf}). We obtain (as submodule of the fusion product) an extremal loop weight module associated to the weight $\Lambda_\ell - \Lambda_0$ (Theorem \ref{thmelwminf}), called the extremal fundamental loop weight module of weight $\Lambda_\ell - \Lambda_0$. Furthermore we explicit the action of $\U_q(\hat{sl}_\infty)$ on it (Theorem \ref{thmformactinf}). Section 4 is devoted to the study of fusion products of extremal fundamental loop weight modules. When the non-zero complex parameters are chosen generic, we get extremal loop weight modules (Theorem \ref{thmgencaseinf}). In the non-generic case we recover all the extremal fundamental loop weight modules from the fusion product of the extremal fundamental loop weight modules associated to the weight $\Lambda_1 - \Lambda_0$ and to a $q$-segment (Theorem \ref{thmtpvrisomeflwm}). By fusion product of $k$ extremal fundamental loop weight modules of weight $\Lambda_\ell - \Lambda_0$, we obtain extremal loop weight modules of the weights $k \Lambda_\ell - k \Lambda_0$ ($\ell, k \geq 1$) (Theorem \ref{thmtpelwminf}). In Section 5, we discuss the applications to quantum toroidal algebras. We show that the $q$--character formulae obtained from the extremal fundamental loop weight $\U_q(\hat{sl}_\infty)$-modules are the $q$--character of representations of the quantum toroidal algebra $\U_q(sl_{n+1}^{tor})$ (Theorem \ref{thmqchaactreptorinf}). Furthermore we determine when these representations give extremal loop weight modules (Theorem \ref{thmelwmtorinf}).

\medskip

\textbf{Acknowledgements :}
I would like to thank my advisor David Hernandez for his encouragements and for his precious comments. I am grateful to Nicolas Jacon for his interest in my work. I want also to thank Alexandre Bouayad and Dragos Fratila for their accurate remarks.

\section{Background}

\subsection{Cartan matrix} \nocite{kac_infinite-dimensional_1990}  Let $C=(C_{i,j})_{i,j\in \ZZ}$ be the Cartan matrix of type $A_\infty$, that is ($i,j \in \ZZ$)
$$ C_{i,i} = 2 \text{ , } C_{i,i+1} = C_{i+1, i} = -1 \text{ and } C_{i,j} = 0 \text{ otherwise. } $$
Consider the vector space
$$\Hlie = \bigoplus_{i \in \ZZ} \QQ h_i$$
and the linear functions $\Lambda_i$ (the fundamental weights) on $\Hlie$ given by 
$$\Lambda_i(h_j) = \delta_{i,j} \text{ for all } i, j \in \ZZ.$$
For $i \in \ZZ$, set $\alpha_i = 2 \Lambda_i - \Lambda_{i-1} - \Lambda_{i+1}$. Denote by $\Pi = \{\alpha_i , i \in \ZZ \}$ the set of simple roots and $\Pi^\vee = \{h_i, i \in \ZZ\}$ the set of simple coroots. Let $P =\bigoplus_{i \in \ZZ} \ZZ \Lambda_i$ be the weight lattice and $P^{+} = \bigoplus_{i \in \ZZ} \NN \Lambda_i$ the semigroup of dominant weights. Let $Q = \bigoplus_{i \in I} \ZZ \alpha_i \subseteq P$ (the root lattice) and $Q^+ = \bigoplus_{i \in I} \NN \alpha_i \subseteq Q$. For $\lambda, \mu \in \Hlie^{\ast}$, write $\lambda \geq \mu$ if $\lambda - \mu \in Q^{+}$.

Denote by $W$ the Weyl group of type $sl_\infty$: it is the subgroup of transformations stabilizing $P$ generated by the simple reflections $s_i$ defined by $s_i(\lambda)=\lambda-\lambda(h_i)\alpha_i$ ($i\in \ZZ, \lambda \in P$).

\subsection{Quantum group $\U_q(sl_\infty)$} In this article $q = e^{t} \in\CC^*$ $(t \in \CC)$ is not a root of unity and is fixed. For $l\in\ZZ, r\geq 0, m\geq m'\geq 0$ we set
$$[l]_q=\frac{q^l-q^{-l}}{q-q^{-1}}\in\ZZ[q^{\pm 1}],\  [r]_q!=[r]_q[r-1]_q\dots[1]_q,\ \begin{bmatrix}m\\m'\end{bmatrix}_q=\frac{[m]_q!}{[m-m']_q![m']_q!}.$$

\begin{defi} The quantum group $\U_q(sl_\infty)$ is the $\CC$-algebra with generators $k_h$ \linebreak $(h \in \Hlie)$, $x_i^{\pm}$ $(i\in \ZZ)$ and relations
\begin{equation*}k_h k_{h'} = k_{h+h'}, \  k_0 = 1,\end{equation*}
\begin{equation*}k_h x_j^{\pm}k_{-h}=q^{\pm \alpha_j(h)}x_j^{\pm},\end{equation*}
\begin{equation*}[x_i^+,x_j^-]=\delta_{i,j}\frac{k_i-k_{i}^{-1}}{q-q^{-1}},\end{equation*}
\begin{equation*}
(x_i^{\pm})^{(2)}x_{i+1}^{\pm} - x_i^{\pm}x_{i+1}^{\pm}x_i^{\pm} + x_{i+1}^{\pm}(x_i^{\pm})^{(2)} = 0,\end{equation*}
\end{defi}

\noindent where we use the notations $k_i^\pm = k_{\pm h_i}$ and for all $r \geq 0$, $(x_i^\pm)^{(r)} = \frac{(x_i^\pm)^r}{[r]_q!}$.

For $J = [a,b] \subset \ZZ$ denote by $\U_q(sl_\infty)_{J}$ the subalgebra of $\U_q(sl_\infty)$ generated by the $x_i^{\pm}, k_{h}$ for $i\in J$, $h \in \oplus_{j \in [a,b]} \QQ h_j$. Then $\U_q(sl_\infty)_{J}$ is isomorphic to the quantum group $\U_q(sl_{b-a+2})$.

Let $\U_q(sl_{\infty})^+$ (resp. $\U_q(sl_{\infty})^-$, $\U_q(\Hlie)$) be the subalgebra of $\U_q(sl_{\infty})$ generated by the $x_i^+$ (resp. the $x_i^-$, resp the $k_h$). We have a triangular decomposition of $\U_q(sl_{\infty})$ (see \cite{lusztig_introduction_1993}): 

\begin{thm} We have an isomorphism of vector spaces
$$\U_q(sl_{\infty}) \simeq \U_q(sl_{\infty})^- \otimes \U_q(\Hlie) \otimes \U_q(sl_{\infty})^+.$$
\end{thm}

\subsection{Representations of $\U_q(sl_\infty)$} For $V$ a representation of $\U_q(sl_\infty)$ and $\nu \in P$, the weight space $V_{\nu}$ of $V$ is
$$V_\nu = \{v\in V|k_i \cdot v = q^{\nu(h_i)}v, \forall i \in \ZZ \}.$$

\begin{defi}\label{defcato} A representation  $V$ is said to be in the category $\mathcal{O}$ if
\begin{enumerate}
\item[(i)] it admits a weight space decomposition $V = \bigoplus_{\nu\in P} V_\nu$,
\item[(ii)] $V_\nu$ is finite-dimensional for all $\nu$,
\item[(iii)] $\{\nu \vert V_\nu \neq \{0\} \} \subset \cup_{j = 1, \cdots, N} \{\nu \vert \nu \leq \lambda_j \}$ for some $\lambda_1, \cdots , \lambda_N \in P$.
\end{enumerate}
\end{defi}

For $\lambda\in P$, a representation $V$ is said to be of highest weight $\lambda$ if there is $v\in V_\lambda$ such that for all $i\in I, x_i^+ \cdot v = 0$ and $\U_q(sl_\infty) \cdot v = V$. Such a representation is in the category $\mathcal{O}$. Furthermore there is a unique simple highest weight module of highest weight $\lambda$.

\begin{defi}\label{defint} A representation  $V$ is said to be integrable if
\begin{enumerate}
\item[(i)] it admits a weight space decomposition $V = \bigoplus_{\nu\in P} V_\nu$,
\item[(ii)] all the $x_i^\pm$ ($i \in \ZZ$) are locally nilpotent.
\end{enumerate}
\end{defi}

\begin{thm}\cite{lusztig_introduction_1993, hernandez_algebra_2011}
The simple highest weight module of highest weight $\lambda$ is integrable if and only if $\lambda$ is dominant. It is denoted $V(\lambda)$.
\end{thm}

Now we recall the definition and some properties of extremal weight modules for the quantum group $\U_q(sl_\infty)$ \cite{kashiwara_crystal_1994}.

\begin{defi} For an integrable $\U_q(sl_\infty)$-module $V$ and $\lambda\in P$, a vector $v\in V_{\lambda}$ is called $i$-extremal if $x_i^+ \cdot v=0$ or $x_i^- \cdot v = 0$. In this case we set $$S_i(v) = (x_i^-)^{(\lambda(h_i))} \cdot v \text{ or } S_i(v) = (x_i^+)^{(-\lambda(h_i))} \cdot v.$$ A vector $v\in V_{\lambda}$ is called extremal of weight $\lambda$ if, for any $l \geq 1$, $S_{i_1} \cdots S_{i_l}(v)$ is $i$-extremal for any $i, i_1, \cdots, i_l \in \ZZ$. \end{defi}

\noindent For $v$ an extremal vector of weight $\lambda$, we set $$W \cdot v = \{S_{i_1} \cdots S_{i_l}(v) \vert l \in \NN, i_1, \cdots i_l \in \ZZ \}.$$

\begin{defi} For $\lambda\in P$, the extremal weight module of extremal weight $\lambda$ is the $\U_q(sl_\infty)$-module generated by a vector $v_{\lambda}$ with the defining relations that $v_{\lambda}$ is extremal of weight $\lambda$. It is denoted $V(\lambda)$.\end{defi}

\begin{ex} If $\lambda$ is dominant, $V(\lambda)$ is the simple highest weight module of highest weight $\lambda$. \end{ex}

\begin{thm} \cite{kashiwara_crystal_1994} For $\lambda\in P$, the module $V(\lambda)$ is integrable and has a crystal basis $\mathcal{B}(\lambda)$.\end{thm}

\subsection{Quantum affinization $\U_q(\hat{sl}_\infty)$}

We recall the definition of the quantum affinization $\U_q(\hat{sl}_\infty)$ (without central charge) in terms of currents.

\begin{defi}\label{defqaainf}
The quantum affinization $\U_q(\hat{sl}_\infty)$ is the $\CC$-algebra with generators $x_{i,r}^\pm$  ($i \in \ZZ, r \in \ZZ$), $k_h$ ($h \in \Hlie$), $h_{i,m}$ ($i \in \ZZ, m \in \ZZ-\{0\}$) and the following defining relations ($i,j \in \ZZ, h, h' \in \Hlie$):
\begin{equation*}
k_h k_{h'} = k_{h+h'}, \  k_0 = 1,
\end{equation*}
\begin{equation*}
\phi^\pm_i(z)\phi^\pm_j (w) = \phi^\pm_j(w)\phi^\pm_i (z) \text{ , } \phi^-_i(z)\phi^+_j (w) = \phi^+_j(w)\phi^-_i (z),
\end{equation*}
\begin{equation}\label{relxcartpl}
(w -q^{\pm C_{ij}}z) \phi_i^+(z)x_j^\pm(w)= 
(q^{\pm C_{ij}}w - z) x_j^\pm(w)\phi_i^+(z),
\end{equation}
\begin{equation}\label{relxcartmo}
(w -q^{\pm C_{ij}}z) \phi_i^-(z)x_j^\pm(w)= 
(q^{\pm C_{ij}}w - z) x_j^\pm(w)\phi_i^-(z),
\end{equation}
\begin{equation*}
[x_i^+(z),x_j^-(w)]=\frac{\delta_{i,j}}{q-q^{-1}}(\delta\bigl(\frac{w}{z}\bigr)\phi_i^+(w) -\delta\bigl(\frac{z}{w}\bigr)\phi_i^-(z)),
\end{equation*}
\begin{equation*}
(z-q^{\pm C_{ij}}w) x_i^\pm(z)x_j^\pm(w)=(q^{\pm a_{ij}}z-w)x_j^\pm(w)x_i^\pm(z),
\end{equation*}
\begin{eqnarray*}
\begin{array}{c}
x_i^\pm(z_1)x_i^\pm(z_2)x_{i\pm1}^\pm(w)-(q+q^{-1})x_i^\pm(z_1)x_{i\pm 1}^\pm(w) x_i^\pm(z_2) \\
+x_{i\pm1}^\pm(w) x_i^\pm(z_1) x_i^\pm(z_2) + (z_1\leftrightarrow  z_2)=0,
\end{array}
\end{eqnarray*}
and $[x_i^\pm(z),x_j^\pm(w)]=0$ for $i \neq j,j \pm 1$.

Here we use the formal power series $\delta(z) = \sum_{s \in \ZZ} z^{s}$ and

\begin{eqnarray*}
x_i^\pm(z) = \sum_{r \in \ZZ} x_{i,r}^\pm z^{r},
\end{eqnarray*}
\begin{eqnarray*}
\phi_i^{\pm}(z) = \sum_{m \geq 0} \phi_{i,\pm m}^\pm z^{\pm m} = k_i^{\pm 1} \exp(\pm(q-q^{-1}) \sum_{m' \geq 1} h_{i, \pm m'} z^{\pm m'}).
\end{eqnarray*}
\end{defi}


There is an algebra morphism $\U_q(sl_\infty)\rightarrow \U_q(\hat{sl}_\infty)$ defined by $k_h\mapsto k_h$, $x_i^{\pm}\mapsto x_{i,0}^{\pm}$ ($h \in \Hlie, i \in \ZZ$). Its image is called the horizontal quantum affine subalgebra of $\U_q(\hat{sl}_\infty)$ and is denoted by $\U_q^{h}(\hat{sl}_\infty)$. In particular, a $\U_q(\hat{sl}_\infty)$-module $V$ has also a structure of a $\U_q(sl_\infty)$-module. We denote by $\mathrm{Res}(V)$ the restricted $\U_q(sl_\infty)$-module obtained from $V$.

For $J = [a,b] \subseteq \ZZ$, let us denote by $ \U_{q}(\hat{sl}_\infty)_J $ the subalgebra of $ \U_{q}(\hat{sl}_\infty) $ generated by the $x_{i,r}^{\pm}$, $k_h$, $h_{i,m}$ ($i\in J$, $r\in\ZZ$, $m\in\ZZ-\{0\}$, $h \in \oplus_{j \in J} \QQ h_j$). It is isomorphic to the quantum affine algebra $\U_q(\hat{sl}_{b-a+2})'$ \cite{beck_braid_1994, drinfeld_new_1988}.

For $i\in \ZZ$, the subalgebra $ \hat{\U_i} $ generated by the $x_{i,r}^{\pm}, h_{i,m}, k_{p h_i}$ ($ r \in \ZZ$, $m \in \ZZ - \{0\}$, $p \in \QQ$) is isomorphic to $ \U_{q}(\hat{sl}_{2})'$.

We have a triangular decomposition of $\U_q(\hat{sl}_\infty)$.

\begin{thm}\label{dtrianinf} \cite{miki_representations_2000, nakajima_quiver_2001} We have an isomorphism of vector spaces
$$\U_q(\hat{sl}_\infty)\simeq \U_q(\hat{sl}_\infty)^-\otimes\U_q(\hat{\Hlie})\otimes\U_q(\hat{sl}_\infty)^+,$$ 
where $\U_q(\hat{sl}_\infty)^{\pm}$ (resp. $\U_q(\hat{\Hlie})$) is generated by the $x_{i,r}^{\pm}$ (resp. the $k_h$, the $h_{i,m}$).
\end{thm}

\subsection{Representations of $\U_q(\hat{sl}_\infty)$}

\begin{defi} A representation $V$ of $\U_q(\hat{sl}_\infty)$ is said to be integrable (resp. in the category $\mathcal{O}$) if $\mathrm{Res}(V)$ is integrable (resp. in the category $\mathcal{O}$) as a $\U_q(sl_\infty)$-module. One denotes by $\mathcal{O}_{\mathrm{int}}$ the category of integrable representations of $\U_q(\hat{sl}_\infty)$ in the category $\mathcal{O}$.
\end{defi}

\begin{defi} A representation $V$ of $\U_q(\hat{sl}_\infty)$ is said to be of $\ell$-highest weight if there is $v\in V$ such that 
\begin{enumerate}
\item[(i)] $V = \U_q(\hat{sl}_\infty)^- \cdot v$, 
\item[(ii)] $\U_q(\hat{\Hlie}) \cdot v=\CC v$, 
\item[(iii)] for any $i\in \ZZ, r\in\ZZ$, $x_{i,r}^+ \cdot v=0$.
\end{enumerate}
\end{defi}

\noindent Such a representation is in the category $\mathcal{O}$. For $\gamma \in \mathrm{Hom}(\U_q(\hat{\Hlie}), \CC)$, by Theorem \ref{dtrianinf} we have a corresponding Verma module $M(\gamma)$ and a simple representation $V(\gamma)$ which are $\ell$-highest weight.

\begin{thm}\label{cond} \cite{hernandez_algebra_2011} The simple representations $V(\gamma)$ of $\U_q(\hat{sl}_\infty)$ are integrable if there is $(P_i)_{i \in \ZZ}\in (1+u\mathbb{C}[u])^{(\ZZ)}$ satisfying for $i\in \ZZ$ the relation in $\mathbb{C}[[z]]$ (resp. in $\mathbb{C}[[z^{-1}]]$)
$$\gamma(\phi_i^\pm(z))=q^{\text{deg}(P_i)}\frac{P_i(zq^{-1})}{P_i(zq)}.$$
\end{thm}

\noindent The polynomials $P_i$ are called Drinfeld polynomials and the representation $V(\gamma)$ is then denoted by $V((P_i)_{i \in \ZZ})$.

The Kirillov-Reshetikhin module associated to $k \geq 0$, $a \in \CC^{\ast}$ and $\ell \in \ZZ$, is the simple integrable representation defined by the $n$--tuple
$$P_i(u) = \left\lbrace \begin{array}{l} (1-ua)(1-uaq^{2}) \cdots (1-uaq^{2(k-1)}) \ \mathrm{for} \ i = \ell, \\ 1 \ \mathrm{for} \ i \neq \ell. \end{array} \right.$$
\noindent If $k=1$, it is also called fundamental module.

\begin{thm}\label{thmlimind}\cite{hernandez_algebra_2011}
Let $v$ be a $\ell$-highest weight vector of the simple $\U_q(\hat{sl}_\infty)$-module $V((P_i)_{i \in \ZZ})$. Then
$$V((P_i)_{i \in \ZZ}) = \bigcup_{n \geq 0} V_n((P_i)_{i \in \ZZ}) \text{ where } V_n((P_i)_{i \in \ZZ}) = \U_q(\hat{sl}_\infty)_{[-n, n]} \cdot v.$$
Furthermore each $V_n((P_i)_{i \in \ZZ})$ is a simple finite-dimensional $\U_q(\hat{sl}_\infty)_{[-n, n]}$-module.
\end{thm}

Consider an integrable representation $V$ of $\U_q(\hat{sl}_\infty)$. As the subalgebra $\U_q(\hat{\Hlie})$ is commutative, the weight spaces $V_{\nu}$ split in simultaneous generalized eigenspaces
$$V_\nu = \bigoplus_{\gamma \in \mathrm{Hom}(\U_q(\hat{\Hlie}), \CC)} V_{\gamma},$$
where $$V_{\gamma} = \{x \in V / \exists p \in \NN, \forall i \in \ZZ, \forall m \geq 0, (\phi_{i, \pm m}^{\pm} - \gamma(\phi_{i, \pm m}^{\pm}))^{p} \cdot x = 0 \}.$$ If $V_{\gamma} \neq \{0\}$, $\gamma$ is called an $\ell$-weight of $V$ and $V_\gamma$ is an $\ell$-weight space.

\begin{defi}
A $\U_q(\hat{sl}_\infty)$-module $V$ is weighted if the Cartan subalgebra $\U_q(\hat{\Hlie})$ acts on $V$ by diagonalizable operators. The module $V$ is thin if it is weighted and the joint spectrum is simple.
\end{defi}

\begin{defi}\cite{frenkel_$q$-characters_1999, hernandez_representations_2005, nakajima_quiver_2001} The $q$--character of an integrable representation $V$ of $\U_q(\hat{sl}_\infty)$ with finite-dimensional $\ell$-weight spaces is defined by the formal sum
$$\chi_q(V) = \sum_{\gamma \in \mathrm{Hom}(\U_q(\hat{\Hlie}), \CC)} \dim(V_{\gamma}) e(\gamma).$$
\end{defi}

As in \cite{feigin_representations_2013}, we set
$$\psi(z)=\frac{q-q^{-1}z}{1-z}.$$

\begin{prop}\cite{frenkel_$q$-characters_1999, hernandez_algebra_2011}
Let $V$ be an integrable representation of $\U_q(\hat{sl}_\infty)$. An $\ell$-weight $\gamma$ of $V$ satisfies the property
\begin{enumerate}
\item for all $i \in \ZZ$, there exist $k,l \in \NN$ and $a_{1,i}, \cdots, a_{k,i}, b_{1,i}, \cdots, b_{l,i} \in \CC^{\ast}$ such that in $\CC[[z]]$ (resp. in $\CC[[z^{-1}]]$):
$$\sum_{m \geq 0} \gamma(\phi_{i, \pm m}^{\pm}) z^{\pm m} = \prod_{1 \leq u \leq k} \psi(a_{u,i}qz) \prod_{1 \leq v \leq l} \psi(b_{v,i}qz)^{-1} = \prod_{1 \leq u \leq k} \psi(a_{u,i}qz) \prod_{1 \leq v \leq l} \psi(b_{v,i}^{-1}qz^{-1})$$
\end{enumerate}
Furthermore if $V$ is in the category $\mathcal{O}$, then
\begin{enumerate}
\item[(2)] there exist $\omega \in P^{+}$, $\alpha \in Q^{+}$ satisfying $\nu = \omega - \alpha$.
\end{enumerate}
\end{prop}

Consider formal variables $ Y_{i,a}^{\pm 1} $ ($ i \in \ZZ $, $ a \in \CC^{\ast}$) and let $ A $ be the group of monomials of the form $$ m = \prod_{i \in \ZZ, a \in \CC^{\ast}} Y_{i,a}^{u_{i,a}(m)}, \ u_{i,a}(m) \in \ZZ.$$
Set $$ A_{i,a} = Y_{i,aq^{l-1}} Y_{i,aq^{l+1}} Y_{i-1,aq^{l}}^{-1}Y_{i+1,aq^{l}}^{-1} \in A.$$
A monomial $ m $ is said to be $J$-dominant ($ J \subset \ZZ $) if for all $ j \in J $ and $a\in \CC^{\ast} $ we have $ u_{j,a}(m) \geq 0 $. A $ \ZZ $-dominant monomial is said to be dominant.

Let $V$ be an integrable $\U_q(\hat{sl}_\infty)$-module. For $\gamma$ an $\ell$-weight of $V$, one defines the monomial $m_{\gamma} = \prod_{i \in \ZZ} \prod_{a_i \in \CC^{\ast}} Y_{i,a_i} \prod_{b_i \in \CC^{\ast}} Y_{i,b_i}^{-1}$ where
$$\sum_{m \geq 0} \gamma(\phi_{i, \pm m}^{\pm}) z^{\pm m} = \prod_{a_i} \psi(a_{i}qz) \prod_{b_i} \psi(b_{i}qz)^{-1}.$$
We denote $V_{\gamma} = V_{m_{\gamma}}$. We rewrite the $q$--character of an integrable representation $V$ with finite-dimensional $\ell$-weight spaces by the formal sum
$$\chi_q(V) = \sum_{m} \dim(V_{m}) m \in \ZZ^{A}.$$
Let us denote by $\mathcal{M}(V)$ the set of monomials occurring in $\chi_q(V)$.

By this correspondence between $\ell$-weights and monomials due to Frenkel-Reshetikhin \cite{frenkel_$q$-characters_1999}, the $\ZZ$-tuple of Drinfeld polynomials are identified with the dominant monomials. In particular for a dominant monomial $m$, one denotes by $V(m)$ the simple module of $\ell$-highest weight $m$. For example $V(Y_{\ell, a}Y_{\ell, aq^{2}} \dots Y_{\ell, aq^{2(k-1)}})$ is the Kirillov-Reshetikhin module associated to $\ell \in \ZZ$, $k \geq 0$ and $a \in \CC^{\ast}$, and $V(Y_{\ell, a})$ is the fundamental module associated to $\ell \in \ZZ$ and $a \in \CC^{\ast}$.

In general no formula is known for the $q$--character of a $\ell$-highest weight module of $\U_q(\hat{sl}_\infty)$. But in the special case of Kirillov-Reshetikhin modules, the $q$--character can be explicitly given as a tableaux sum expression. For that we set for all $a \in \CC^{\ast}$ and $j \in \ZZ$
$$\ffbox{j}_a = Y_{j-1, aq^{j}}^{-1}Y_{j, aq^{j-1}}.$$
For $I, J \subset \ZZ$, a tableaux $(T_{i, j})_{i \in I, j \in J}$ is said to be semi-standard if it has coefficients in $\ZZ$ which increase relatively to $j$ and strictly increase relatively to $i$.

Consider $\ell \in \ZZ$ and $k \geq 0$. Let $\mathcal{T}_{\ell, k}$ be the set of semi-standard tableaux $(T_{i, j})_{i \leq \ell, 1 \leq j \leq k}$ such that for any $1 \leq j \leq k$,
$$T_{i, j} = i \text{ for } i << 0.$$

\begin{thm}\cite{hernandez_algebra_2011}
We have the following formula
\begin{align}\label{formqchainfty}
\chi_q(V(Y_{\ell, a} \cdots Y_{\ell, aq^{2(k-1)}})) = \sum_{T \in \mathcal{T}_{\ell, k}} m_T
\end{align}
where $m_T = \prod_{i \leq \ell, 1 \leq j \leq k} \ffbox{T_{i,j}}_{aq^{\ell-1+2(j-i)}}$.
\end{thm}

\subsection{Extremal loop weight modules}

The definition of extremal loop weight modules is given in \cite{hernandez_algebra_2011} for the quantum toroidal algebras and studied in \cite{mansuy_quantum_2012, mansuy_extremal_2013}. We introduce its definition in the context of the quantum affinization $ \mathcal{U}_{q}(\hat{sl}_\infty) $.

\begin{defi}\label{defelminf} An extremal loop weight module of $ \mathcal{U}_{q}(\hat{sl}_\infty) $ of $\ell$-weight $m$ is an integrable representation $ V $ such that there is $ v \in V_m $ satisfying
\begin{enumerate}
\item[(i)] $ \mathcal{U}_{q}(\hat{sl}_\infty) \cdot v = V $,
\item[(ii)] $ v $ is extremal for $ \mathcal{U}_{q}^{h}(\hat{sl}_\infty) $,
\item[(iii)] $ \mathcal{U}_{q}(\hat{sl}_\infty)_J \cdot w $ is finite-dimensional for all $ w \in V $ and $ J = [a, b] \subset \ZZ $ finite.
\end{enumerate} \end{defi}

\begin{prop}
Let $m$ be a dominant monomial. Then the simple $\ell$-highest weight module $V(m)$ is an extremal loop weight module of $\U_{q}(\hat{sl}_\infty)$.
\end{prop}

\begin{proof}
Let $v$ be a $\ell$-highest weight vector of $V(m)$. Denote its weight by $\lambda \in P^+$. As it is recalled above, $V(m) = \U_{q}(\hat{sl}_\infty) \cdot v$  and is integrable. Furthermore $v$ is a highest weight vector for $\U_{q}^h(\hat{sl}_\infty)$ which satisfies for all $i \in \ZZ$,
$$(x_i^{-})^{1+\lambda(h_i)} \cdot v = 0.$$
Hence $\U_{q}^h(\hat{sl}_\infty) \cdot v$ is isomorphic to the simple highest weight $\U_q(sl_{\infty})$-module $V(\lambda)$, and $v$ is extremal of weight $\lambda$ for the horizontal quantum affine subalgebra.

It remains to show that for all $w \in V(m)$ and $J = [a, b] \subset \ZZ$ finite, $\U_q(\hat{sl}_\infty)_J \cdot w$ is finite-dimensional. But there exists $n \in \NN$ such that $J \subset [-n , n]$ and
\begin{align*}
\U_q(\hat{sl}_\infty)_J \cdot w \subset \U_q(\hat{sl}_\infty)_{[-n, n]} \cdot w \subset V_n(m) = \U_q(\hat{sl}_\infty)_{[-n, n]} \cdot v.
\end{align*}
As $V_n((P_i)_{i \in \ZZ})$ is finite-dimensional by Theorem \ref{thmlimind}, it is also the case of $\U_q(\hat{sl}_\infty)_J \cdot w$. So the $\U_q(\hat{sl}_\infty)$-module $V(m)$ is an extremal loop weight module of $\ell$-weight $m$.
\end{proof}

In this paper we construct extremal loop weight modules of $\U_q(\hat{sl}_\infty)$ by fusion product of $\ell$-highest weight modules and $\ell$-lowest weight modules. This process is introduced in \cite{mansuy_extremal_2013} in the toroidal case, and is motivated by the original study of extremal weight modules of quantum Kac-Moody algebras \cite{kashiwara_crystal_1994}: in fact to study these representations, tensor products of highest weight modules and lowest weight modules are considered (see \cite[Lemma 8.2.1]{kashiwara_crystal_1994}).

\subsection{Drinfeld coproduct}

Let $\Delta$ be the coproduct defined by ($i\in \ZZ$)
\begin{equation}\label{deltaxplus}
\Delta(x_i^+(z))=x_i^+(z)\otimes 1 +\phi_i^-(z)\otimes x_i^+(z),
\end{equation}
\begin{equation}\label{deltaxmoins}
\Delta(x_i^-(z))= x_i^-(z)\otimes \phi_i^+(z)+1\otimes x_i^-(z),
\end{equation}
\begin{equation}\label{deltacart}
\Delta(\phi_i^\pm(z))=\phi_i^\pm(z)\otimes \phi_i^\pm (z).
\end{equation}

\noindent Note that this map does not define a coproduct in the usual sense because (\ref{deltaxplus}) and (\ref{deltaxmoins}) involve infinite sums and are not elements of $\U_q(\hat{sl}_\infty) \otimes \U_q(\hat{sl}_\infty)$ (they take values in a completion of it). However it can be used to define a structure of $\U_q(\hat{sl}_\infty)$-module on (a subspace of) tensor products of $\U_q(\hat{sl}_\infty)$-modules when the summations are well-defined.

\begin{lem}\cite{feigin_representations_2013} 
\label{lemsubminf}
Let $V$, $W$ be $\U_q(\hat{sl}_\infty)$-modules. Let $U\subseteq V\otimes W$ be a linear subspace such that 
for any $u\in U$, $\Delta(x_{i}^{+}(z)) \cdot u$ is well-defined in $U$ ( $i\in \ZZ$). Then the coproduct $\Delta$ endows $U$ with a structure of $\U_q(\hat{sl}_\infty)$-module, called the fusion product of $V$ and $W$.
\end{lem}

\begin{lem}\label{tpalgv}
Assume further that $V$ and $W$ satisfies $\U_{q}(\hat{sl}_\infty)_J \cdot v$ and $\U_{q}(\hat{sl}_\infty)_J \cdot w$ are finite-dimensional for all $v \in V, w \in W$ and all $J = [a, b] \subset \ZZ$ finite in the last lemma. Then the $\U_{q}(\hat{sl}_\infty)$-module $U$ satisfies also
$$\U_{q}(\hat{sl}_\infty)_J \cdot u \text{ is finite-dimensional for all } u \in U \text{ and } J = [a, b] \subset \ZZ \text{ finite.}$$
In particular if $V$ and $W$ are integrable, $U$ is an integrable $\U_q(\hat{sl}_\infty)$-module.
\end{lem}

\begin{proof}
Set $J = [a, b] \subset \ZZ$ finite and let $u = v \otimes W \in V \otimes W$. By the formulas of the coproduct, we have
$$\U_{q}(\hat{sl}_\infty)_J \cdot u \subset \left( \U_{q}(\hat{sl}_\infty)_J \cdot v \right) \otimes \left( \U_{q}(\hat{sl}_\infty)_J \cdot v \right).$$
The result follows.
\end{proof}

\section{Extremal fundamental loop weight modules}

In this section, we define a $\U_q(\hat{sl}_\infty)$-module structure on the tensor product of fundamental modules
$$V(Y_{\ell, a}) \otimes V(Y_{0, aq^{\ell}}^{-1}) \text{ with } \ell \geq 1 \text{ and } a \in \CC^{\ast}$$
by using the Drinfeld coproduct $\Delta$. We obtain in that way new integrable modules $V(Y_{\ell, a}Y_{0, aq^{\ell}}^{-1})$ with basis labelled by semi-standard tableaux of shape $(\ell)$. We show in particular that these modules are $\ell$-extremal. We call them the extremal fundamental loop weight modules of $\U_q(\hat{sl}_\infty)$.

The fusion product process used here is inspired to the one in \cite{mansuy_extremal_2013} for the toroidal case. As it is said in the introduction, the representation theory of $\U_q(\hat{sl}_\infty)$ and of quantum toroidal algebras are very different: the fundamental $\U_q(\hat{sl}_\infty)$-modules are thin, which is not the case for the quantum toroidal algebras $\U_q(sl_{n+1}^{tor})$ (see \cite[Section 4.1]{hernandez_algebra_2011}). Let us point out that these differences have consequences in our work: we obtain here more precise results for the fusion product of $\U_q(\hat{sl}_\infty)$-modules. Another difficulty met in \cite{mansuy_extremal_2013} is that the cyclicity of the Dynkin diagram associated to $\U_q(sl_{n+1}^{tor})$ bring some restricted conditions in the construction of extremal loop weight modules (see \cite[Proposition 3.11]{mansuy_extremal_2013}). We will see that it is possible here to associate extremal fundamental loop weight $\U_q(\hat{sl}_\infty)$-modules for all the nodes of the Dynkin diagram of type $A_\infty$.

In the first part, we give explicit formulae of the action on the fundamental \linebreak $\U_q(\hat{sl}_\infty)$-modules, using the $\U_q(sl_\infty)$-crystal structure on the set of semi-standard \linebreak tableaux. In the second part, we construct the fusion product of the $\U_q(\hat{sl}_\infty)$-modules $V(Y_{\ell, a})$ and $V(Y_{0, aq^{\ell}}^{-1})$, the action of $\U_q(\hat{sl}_\infty)$ being defined via the Drinfeld coproduct (Proposition \ref{tpexistudinf}). We consider a submodule $V(Y_{\ell, a}Y_{0, aq^{\ell}}^{-1})$ of it, which we call the extremal fundamental loop weight module. We study the $\U_q(\hat{sl}_\infty)$-modules $V(Y_{\ell, a}Y_{0, aq^{\ell}}^{-1})$ in the third part: these representations are integrable with basis labelled by semi-standard tableaux of shape $(\ell)$. The action of $\U_q(\hat{sl}_\infty)$ on it is described by the $\U_q(sl_\infty)$-crystal structure on the set of semi-standard tableaux of shape $(\ell)$ (Theorem \ref{thmformactinf}). Furthermore $V(Y_{\ell, a}Y_{0, aq^{\ell}}^{-1})$ is an extremal loop weight module of $\ell$-weight $Y_{\ell, a}Y_{0, aq^\ell}^{-1}$ (Theorem \ref{thmelwminf}) which is irreducible as a $\U_q(\hat{sl}_\infty)$-module and as a $\U_q^h(\hat{sl}_\infty)$-module (Proposition \ref{felwmirredinf}).

\subsection{Action on the fundamental $\U_q(\hat{sl}_\infty)$-modules}

In this section we give explicit formulae of the action on the fundamental $\U_q(\hat{sl}_\infty)$-modules $V(Y_{\ell, a})$ ($\ell \in \ZZ, a \in \CC^{\ast}$).

Let $\mathcal{T}_{\ell}^+ = \mathcal{T}_{\ell, 1}$ be the set of semi-standard tableaux
$$T = ( \cdots < i_{\ell - 2} < i_{\ell-1} < i_\ell ) \text{ such that } i_j = j \text{ for } j <<0.$$
We endow $\mathcal{T}_{\ell}^+$ with a structure of $\U_q(sl_\infty)$-crystal: the weight function is defined by setting
$$\wt(T) = \sum_{j \leq \ell} \Lambda_{i_j} - \Lambda_{i_{j}-1}$$
 for all $T = ( \cdots < i_{\ell - 2} < i_{\ell-1} < i_\ell ) \in \mathcal{T}_\ell^+$. The Kashiwara operators $\tilde{e}_i$, $\tilde{f}_i$ are described in terms of tableaux: for $i \in \ZZ$ we have $\tilde{e}_i \cdot T = T'$ or $0$. Here $T'$ is obtained from $T$ by replacing $i+1$ by $i$. If it is not possible (i.e. when we have both $i+1$ and $i$ in $T$ or when $i+1$ does not occur in $T$), then it is zero. Similarly $\tilde{f}_i \cdot m_{T;j}= m_{T'';j}$ or $0$, where $T''$ is given by replacing $i$ by $i+1$.

\begin{prop}
Let $\ell \in \ZZ$ and $a \in \CC^{\ast}$ and consider the fundamental $\U_q(\hat{sl}_\infty)$-module $V(Y_{\ell,a})$. Then there exists a basis of $V(Y_{\ell,a})$ indexed by $\mathcal{T}_\ell^+$ such that the action of $\U_q(\hat{sl}_\infty$) is given by (for convenience in the calculations, we remind the complex number $a \in \CC^{\ast}$ in subscript)
\begin{eqnarray}\label{actonfundinf}
\begin{array}{rcl}
x_i^+(z) \cdot T_a &=& \displaystyle \sum_{p} \delta_{\{i_p=i+1\}} \delta(aq^{i+\ell + 1-2p}z) (\tilde{e}_i \cdot T)_a,\\
x_{i}^-(z) \cdot T_a &=& \displaystyle \sum_{p} \delta_{\{i_p=i\}} \delta(aq^{i + \ell + 1-2p}z) (\tilde{f}_i \cdot T)_a,\\
\phi_{i}^\pm(z) \cdot T_a &=& \begin{cases} \psi(aq^{i + \ell + 3 -2p}z)^{-1} \cdot T_a & \begin{array}{c} \text{ if there exists } p \leq \ell \text{ such that } \\ i_p=i+1 \text{ and } i_{p-1} \neq i,
\end{array}\\
\psi(aq^{i + \ell + 1-2p}z) \cdot T_a & \begin{array}{c} \text{ if there exists } p \leq \ell \text{ such that } \\  \text{ if } i_p=i \text{ and } i_{p+1} \neq i+1,
\end{array}\\
T_a & \text{ otherwise.} \end{cases}
\end{array}
\end{eqnarray}
where $T = ( \cdots < i_{\ell - 2} < i_{\ell-1} < i_\ell )$ is a semi-standard tableaux in $\mathcal{T}_\ell^+$.
\end{prop}

\noindent In particular for all $T \in \mathcal{T}_\ell^+$, $T_a$ is an $\ell$-weight vector of weight $m_T$.

\begin{proof}
This result is known for the fundamental $\U_q(\hat{sl}_{n+1})$-modules (see \cite[Theorem 4.11]{mansuy_quantum_2012}). We obtain the result for $\U_q(\hat{sl}_\infty)$ by inductive limit, using Theorem \ref{thmlimind}.
\end{proof}

We have an analogue result for the $\ell$-lowest weight modules. For that let $\mathcal{T}_s^-$ ($s \in \ZZ$) be the set of semi-standard tableaux
$$T = (i_{s+1} < i_{s + 2} < i_{s + 3} < \cdots ) \text{ such that } i_{j} = j \text{ for } j >> 0.$$ Then $(\mathcal{T}_s^-, \wt, \tilde{e}_i, \tilde{f}_i)$ is a $\U_q(sl_\infty)$-crystal. And for all $a \in \CC^{\ast}$, there exists a basis of $V(Y_{s, a}^{-1})$ indexed by $\mathcal{T}_s^-$ such that the action of $\U_q(\hat{sl}_\infty)$ on it is given for all \linebreak $T =  (i_{s+1} < i_{s + 2} < i_{s + 3} < \cdots ) \in \mathcal{T}_s^-$ by
\begin{eqnarray*}
\begin{array}{rcl}
x_i^+(z) \cdot T_{a} &=& \displaystyle \sum_{p} \delta_{\{i_p=i+1\}} \delta(aq^{i+s + 1-2p}z) (\tilde{e}_i \cdot T)_{a},\\
x_{i}^-(z) \cdot T_{a} &=& \displaystyle \sum_{p} \delta_{\{i_p=i\}} \delta(aq^{i + s + 1-2p}z) (\tilde{f}_i \cdot T)_{a},\\
\phi_{i}^\pm(z) \cdot T_{a} &=& \begin{cases} \psi(aq^{i + s + 3 -2p}z)^{-1} \cdot T_{a} & \begin{array}{c} \text{ if there exists } p \geq s+1 \text{ such that } \\ i_p=i+1 \text{ and } i_{p-1} \neq i,
\end{array}\\
\psi(aq^{i + s + 1-2p}z) \cdot T_{a} & \begin{array}{c} \text{ if there exists }  p \geq s+1 \text{ such that } \\  \text{ if } i_p=i \text{ and } i_{p+1} \neq i+1,
\end{array}\\
T_{a} & \text{ otherwise.} \end{cases}
\end{array}
\end{eqnarray*}

\subsection{Fusion product of fundamental modules}

\begin{prop}\label{tpexistudinf}
Let $\ell \geq 1$ and $a \in \CC^{\ast}$. The coproduct $\Delta$ is well-defined on the tensor product $V(Y_{\ell, a}) \otimes V(Y_{0, aq^{\ell}}^{-1})$ and it endows it with a structure of integrable \linebreak $\U_q(\hat{sl}_\infty)$-module.
\end{prop}

\begin{rem}
We have introduced a fusion product process in \cite[Section 3]{mansuy_extremal_2013} to define new integrable modules for the quantum toroidal algebras. This process can also be used in our situation to show Proposition \ref{tpexistudinf}. We give here a direct proof of it for the convenience of the reader. As it is pointed out above, the computations are simpler in that case, stemming from the thin property of the fundamental $\U_q(\hat{sl}_\infty)$-modules.
\end{rem}

\begin{proof}
Let us consider
\begin{align*}
T = ( \cdots < i_{\ell-2} < i_{\ell-1} < i_{\ell}) \in \mathcal{T}_\ell^+ \text{ and } T' = (j_1 < j_2 < j_3 < \cdots ) \in \mathcal{T}_0^-.
\end{align*}
We study the action of $x_i^{+}(z)$ on $T_a \otimes T'_{aq^{\ell}} \in V(Y_{\ell, a}) \otimes V(Y_{0, aq^{\ell}}^{-1})$:
\begin{align*}
\begin{array}{c}
x_i^+(z) \cdot (T_a \otimes T'_{aq^{\ell}}) = \sum_{r} \delta_{\{i_r=i+1\}} \delta(a q^{i + \ell + 1 - 2r}z) (\tilde{e}_i \cdot T)_a \otimes T'_{aq^{\ell}} + \\
\displaystyle \sum_{r, s} \left[ 1 + \delta_{\{i_r=i+1\}}\delta_{\{i_{r-1} \neq i\}}\left( \psi (q^{2(s-r)+2} )^{-1} - 1 \right) + \delta_{\{i_r=i\}}\delta_{\{i_{r+1} \neq i+1\}} \left( \psi ( q^{2(s-r)} ) - 1 \right) \right] \\
\times \delta_{\{j_s=i+1\}} \delta(a q^{i + \ell + 1 - 2r}z) T_a \otimes (\tilde{e}_i \cdot T')_{aq^{\ell}}.
\end{array}
\end{align*}
By definition of $\mathcal{T}_\ell^+$ and $\mathcal{T}_0^-$, we have $r \leq i$, $s \geq i+1$, and $s-r \geq 1$. Hence the coefficients of the series $x_i^+(z) \cdot (T_a \otimes T'_{aq^{\ell}})$ are well-defined. The result follows by Lemma \ref{lemsubminf}.
\end{proof}

Let us fix $\ell \in \ZZ$. For $T = ( \cdots < i_{\ell-2} < i_{\ell-1} < i_\ell )$ a semi-standard tableaux in $\mathcal{T}_\ell^+$, we define
$$\alpha_T = \sup \{ p \leq \ell \vert i_p = p \}.$$
Note that $\alpha_T$ is well-defined by definition of $\mathcal{T}_\ell^+$. In the same way, we define for all $T = ( i_{1} < i_{2} < i_{3} < \cdots ) \in \mathcal{T}_0^-$
$$\beta_T = \inf \{ p \geq 1 \vert i_p = p \}.$$

Assume that $\ell \geq 1$. We have endowed the sets $\mathcal{T}_{\ell}^+$ and $\mathcal{T}_0^-$ with a structure of $\U_q(sl_\infty)$-crystal. By the tensor product rules (see \cite{kashiwara_crystal_1994, kashiwara_bases_2002}), $\mathcal{T}_\ell^+ \otimes \mathcal{T}_0^-$ is a \linebreak $\U_q(sl_\infty)$-crystal. Let us denote by $\mathcal{T}_\ell$ the subset of $\mathcal{T}_\ell^+ \otimes \mathcal{T}_0^-$ consisting of
$$T \otimes T' \text{ with } T \in \mathcal{T}_\ell^{+}, T' \in \mathcal{T}_0^- \text{ and } \alpha_T \geq \beta_{T'}-1.$$

\begin{prop}
The set $\mathcal{T}_\ell$ is the connected sub-$\U_q(sl_\infty)$-crystal of $\mathcal{T}_\ell^+ \otimes \mathcal{T}_0^-$ generated by $(\cdots < \ell-2 < \ell-1 < \ell) \otimes (1 < 2 < 3 < \cdots)$.
\end{prop}

\begin{proof}
Set $T \otimes T' \in \mathcal{T}_\ell$ and $i \in \ZZ$. Let us show that $\tilde{e}_i \cdot (T \otimes T') \in \mathcal{T}_\ell \cup \{0\}$ (the proof for the Kashiwara operators $\tilde{f}_i$ is analogous). We have the following equalities
\begin{align*}
\alpha_{\tilde{e}_i \cdot T} & = \begin{cases} \alpha_{T} & \text{ if } \tilde{e}_i \cdot T' \neq 0 \text{ and } i > \alpha_{T}+1, \\
\alpha_{T} + 1 & \text{ if } i = \alpha_{T}+1,
 \end{cases}\\
\beta_{\tilde{e}_i \cdot T'} & = \begin{cases} \beta_{T'} & \text{ if } \tilde{e}_i \cdot T' \neq 0 \text{ and } i < \beta_{T'}-1, \\
\beta_{T'} + 1 & \text{ if } i = \beta_{T'}-1.
 \end{cases}
\end{align*}
In particular $\alpha_{\tilde{e}_i \cdot T} \leq \alpha_T$. So $\tilde{e}_i \cdot (T \otimes T')$ belongs to $\mathcal{T}_\ell \cup \{0\}$, except perhaps when $\alpha_T = \beta_{T'}-1 = i$. But in that case, we have by the tensor product rules
\begin{align*}
\tilde{e}_i \cdot (T \otimes T') = (\tilde{e}_i \cdot T) \otimes T' = 0 \in \mathcal{T}_\ell \cup \{0\}.
\end{align*}
Hence $\mathcal{T}_\ell$ is a sub-$\U_q(sl_\infty)$-crystal of $\mathcal{T}_\ell^+ \otimes \mathcal{T}_0^-$. And one can show by straightforward computations that it is the connected sub-$\U_q(sl_\infty)$-crystal of $\mathcal{T}_\ell^+ \otimes \mathcal{T}_0^-$ generated by $$(\cdots < \ell-2 < \ell-1 < \ell) \otimes (1 < 2 < 3 < \cdots).$$
\end{proof}

Denote by $V(Y_{\ell, a}Y_{0, aq^\ell}^{-1})$ ($\ell \geq 1, a \in \CC^{\ast}$) the subvector space of $V(Y_{\ell, a}) \otimes V(Y_{0, aq^\ell}^{-1})$ generated by $T_a \otimes T'_{aq^{\ell}}$ with $T \otimes T' \in \mathcal{T}_\ell$.

\begin{prop}
The coproduct $\Delta$ endows $V(Y_{\ell, a}Y_{0, aq^\ell}^{-1})$ with a structure of \linebreak $\U_q(\hat{sl}_\infty)$-module. We call it the extremal fundamental loop weight module.
\end{prop}

\begin{proof}
The action of $x_i^\pm(z)$ ($i \in \ZZ$) is well-defined on $V(Y_{\ell, a}Y_{0, aq^\ell}^{-1})$ by Proposition \ref{tpexistudinf}. By Lemma \ref{lemsubminf}, it remains to show that $x_i^+(z) \cdot (T_a \otimes T'_{aq^{\ell}}) \in V(Y_{\ell, a}Y_{0, aq^\ell}^{-1})$ for all $T \otimes T' \in \mathcal{T}_\ell$ (the proof being similar for $x_i^-(z)$). As in the last proof, it holds except perhaps when $\alpha_T = \beta_{T'}-1 = i$. In that case we have
\begin{align*}
x_i^+(z) \cdot T_a \otimes T'_{aq^{\ell}} = \psi(q^2) \delta(aq^{\ell-i-1}z) \cdot T_a \otimes (\tilde{e}_i \cdot T')_{aq^{\ell}} = 0 \in V(Y_{\ell, a}Y_{0, aq^\ell}^{-1})
\end{align*}
Hence the coproduct  $\Delta$ endows $V(Y_{\ell, a}Y_{0, aq^\ell}^{-1})$ with a structure of $\U_q(\hat{sl}_\infty)$-module.
\end{proof}

\subsection{Study of the extremal fundamental loop weight module $V(Y_{\ell, a}Y_{0, aq^\ell}^{-1})$}

Let $\mathcal{T}_{[1, \ell]}$ be the set of semi-standard tableaux $T = (i_1 < i_2 < \cdots < i_\ell)$ of shape $(\ell)$. We consider the application $\Pi : \mathcal{T}_\ell \rightarrow \mathcal{T}_{[1, \ell]}$ defined for all $T \otimes T' \in \mathcal{T}_\ell$ by
$$\Pi(T \otimes T') = (j_1 < j_2 < \cdots < j_{\alpha_T} < i_{\alpha_T + 1} < \cdots < i_\ell)$$
where $T = (\cdots < i_{\ell-2} < i_{\ell-1} < i_{\ell})$ and $T' = (j_1 < j_2 < j_3 < \cdots )$.
This application is bijective, its inverse sending $T = (i_1 < i_2 < \cdots < i_\ell ) \in \mathcal{T}_{[1, \ell]}$ on
$$( \cdots < 0 < \max(i_1, 1) < \cdots < \max(i_\ell, \ell)) \otimes (\min(i_1, 1) < \cdots < \min(i_\ell, \ell) < \ell+1 < \cdots).$$

\begin{prop}
The application $\Pi : \mathcal{T}_\ell \rightarrow \mathcal{T}_{[1, \ell]}$ is an isomorphism of $\U_q(sl_\infty)$-crystals.
\end{prop}

\begin{proof}
Let  $T = (\cdots < i_{\ell-2} < i_{\ell-1} < i_{\ell})$ and $T' = (j_1 < j_2 < j_3 < \cdots )$ be semi-standard tableaux such that $T \otimes T' \in \mathcal{T}_\ell$, and $i \in \ZZ$. We have to consider the following cases
\begin{itemize}
\item[-] $ i > \alpha_T$: then we have $\beta_{T'} \leq i$ and $i, i+1 \in T'$. By the tensor product rules, we have $$\tilde{f}_i \cdot (T \otimes T') = (\tilde{f}_i \cdot T) \otimes T'.$$ Hence,
\begin{align*}
\Pi(\tilde{f}_i \cdot (T \otimes T')) &= \begin{cases} (j_1 < \cdots < i_{p} + 1 < \cdots < i_\ell) \begin{array}{c}
\text{ if there exists } p > \alpha_T \text{ such that }\\
  i_p = i \text{ and } i_{p+1} \neq i+1,
\end{array}\\
0 \text{ otherwise}
\end{cases}\\
&= \tilde{f}_i \cdot \Pi(T \otimes T')
\end{align*}
\item[-] $i < \alpha_T$: then we have $i, i+1 \in T$ and by the tensor product rules, $$\tilde{f}_i \cdot (T \otimes T') = T \otimes (\tilde{f}_i \cdot T').$$ Furthermore, $j_{\alpha_T} \leq j_{\alpha_T+1}-1 = \alpha_T$. Hence, 
\begin{align*}
\Pi(\tilde{f}_i \cdot (T \otimes T')) &= \begin{cases} (j_1 < \cdots < j_{p} + 1 < \cdots < i_\ell)  \begin{array}{c}
\text{ if there exists } 1 \leq p < \alpha_T \text{ such that }\\
 i_p = i \text{ and } i_{p+1} \neq i+1,
\end{array} \\
0  \text{ otherwise}
\end{cases}\\
&= \tilde{f}_i \cdot \Pi(T \otimes T')
\end{align*}
\item[-] $i = \alpha_T$: then we have $i+1 \notin T$ and $\beta_{T'} \leq i+1$.

Assume that $\beta_{T'} < i+1$. In that case, $i, i+1 \in T'$ and we have $$\tilde{f}_i \cdot (T \otimes T') = (\tilde{f}_i \cdot T) \otimes T'.$$ Furthermore $\alpha_{\tilde{f}_i \cdot T} = \alpha_{T} - 1$ and $j_{\alpha_T} = i$. So we have
\begin{align*}
\Pi(\tilde{f}_i \cdot (T \otimes T')) &= (j_1 < \cdots < j_{\alpha_T-1} < i+1 < i_{\alpha_T + 1} < \cdots < i_\ell) \\
&= \tilde{f}_i \cdot (j_1 < \cdots < j_{\alpha_T-1} < i = j_{\alpha_T} < i_{\alpha_T + 1} < \cdots < i_\ell) \\
&= \tilde{f}_i \cdot \Pi(T \otimes T').
\end{align*}

If $\beta_{T'} = i+1$, we have $i \notin T'$. By the tensor product rules we get
$$\tilde{f}_i \cdot (T \otimes T') = T \otimes (\tilde{f}_i \cdot T') = 0.$$
Furthermore we have $\alpha_T = \beta_{T'}-1$ and $j_{\alpha_T} < \alpha_T = i < i_{\alpha_T+1}$. Then \linebreak $i \notin \Pi(T \otimes T')$ and $\tilde{f}_i \cdot \Pi(T \otimes T') = 0$.
\end{itemize}
As the application $\Pi : \mathcal{T}_\ell \rightarrow \mathcal{T}_{[1, \ell]}$ is also bijective, this is an isomorphism of \linebreak $\U_q(sl_\infty)$-crystals.
\end{proof}

\begin{rem}\label{remlienmoncriinf}
Let us describe this isomorphism at the level of the monomial crystals (see \cite{kashiwara_realizations_2003, mansuy_quantum_2012, nakajima_$t$-analogs_2003} for its definition). For that denote by $\mathcal{M}_\ell$ the sub-$\U_q(sl_\infty)$-crystal of $\mathcal{M}(Y_{\ell, a}) \otimes \mathcal{M}(Y_{0, aq^{\ell}}^{-1})$ generated by $Y_{\ell, a} \otimes Y_{0, aq^{\ell}}^{-1}$. Let us consider also $\mathcal{M}(Y_{\ell, a}Y_{0, aq^{\ell}}^{-1})$ the monomial crystal generated by $Y_{\ell, a}Y_{0, aq^{\ell}}^{-1}$. By that we have done above, we have
\begin{align*}
\mathcal{M}(Y_{\ell, a}Y_{0, aq^{\ell}}^{-1}) = \lbrace m_T = \prod_{1 \leq j \leq \ell} \ffbox{i_j}_{aq^{\ell+1-2j}} \ \vert \ T = (i_1 < i_2 < \cdots < i_\ell) \in \mathcal{T}_{[1, \ell]} \rbrace
\end{align*}
and
\begin{align*}
\mathcal{M}_\ell = \lbrace m_T \otimes m_{T'} \vert T \otimes T' \in \mathcal{T}_\ell \rbrace.
\end{align*}
In the monomial realization, $\Pi$ is simply given by the product in $A$
$$\Pi(m_1 \otimes m_2) = m_1 \times m_2 \text{ with } m_1 \otimes m_2 \in \mathcal{M}_\ell.$$
\end{rem}

As a direct consequence, we have

\begin{prop}
Let $\ell \geq 1$ and $a \in \CC^{\ast}$. The $\U_q(\hat{sl}_\infty)$-module $V(Y_{\ell, a}Y_{0, aq^\ell}^{-1})$ is thin, and we have
\begin{equation}\label{qchainf}
\chi_q(V(Y_{\ell, a}Y_{0, aq^\ell}^{-1})) = \sum_{T \in \mathcal{T}_{[1, \ell]}} m_T
\end{equation}
where $m_T = \displaystyle \prod_{1 \leq j \leq \ell} \ffbox{i_j}_{aq^{\ell+1-2j}}$.
\end{prop}

\begin{proof}
Let us recall that we have a basis of $\ell$-weight vectors of $V(Y_{\ell, a}Y_{0, aq^\ell}^{-1})$: it is given by the $T_a \otimes T'_{aq^{\ell}}$ in $ \mathcal{T}_\ell$. Furthermore by (\ref{deltacart}) and the remark above, the $\ell$-weight of the vector $T_a \otimes T'_{aq^{\ell}}$ is
$$m_T \times m_{T'} = m_{\Pi(T \otimes T')}.$$
Hence  we have
$$\chi_q(V(Y_{\ell, a}Y_{0, aq^\ell}^{-1})) = \sum_{T \in \mathcal{T}_{[1, \ell]}} m_T.$$
\end{proof}

\begin{rem}
In particular we have shown here that the monomial crystal $\mathcal{M}(Y_{\ell, a}Y_{0, aq^{\ell}}^{-1})$ generated by $Y_{\ell, a}Y_{0, aq^{\ell}}^{-1}$ is closed in the sense of \cite[Definition 3.6]{mansuy_quantum_2012}.
\end{rem}

By the bijection $\Pi$, we parametrize the basis of $V(Y_{\ell, a}Y_{0, aq^\ell}^{-1})$ by the $\U_q(sl_\infty)$-crystal $\mathcal{T}_{[1, \ell]}$. Let us explicit formulae of the action on it.

\begin{thm}\label{thmformactinf}
The action of $\U_q(\hat{sl}_\infty)$ on $V(Y_{\ell, a}Y_{0, aq^\ell}^{-1})$ is given for all \linebreak $T = (i_1 < i_2 < \cdots < i_\ell) \in \mathcal{T}_{[1, \ell]}$ by
\begin{eqnarray}\label{formactinf}
\begin{array}{rcl}
x_i^+(z) \cdot T_a &=& \displaystyle \sum_{1 \leq p \leq \ell} \delta_{\{i_p=i+1\}} \delta(aq^{i+\ell + 1-2p}z) (\tilde{e}_i \cdot T)_a,\\
x_{i}^-(z) \cdot T_a &=& \displaystyle \sum_{1 \leq p \leq \ell} \delta_{\{i_p=i\}} \delta(aq^{i + \ell + 1-2p}z) (\tilde{f}_i \cdot T)_a,\\
\phi_{i}^\pm(z) \cdot T_a &=& \begin{cases} \psi(aq^{i + \ell + 3 -2p}z)^{-1} \cdot T_a & \begin{array}{c} \text{ if there exists } 1 \leq p \leq \ell \text{ such that } \\ i_p=i+1 \text{ and } i_{p-1} \neq i,
\end{array}\\
\psi(aq^{i + \ell + 1-2p}z) \cdot T_a & \begin{array}{c} \text{ if there exists } 1 \leq p \leq \ell \text{ such that } \\  \text{ if } i_p=i \text{ and } i_{p+1} \neq i+1,
\end{array}\\
T_a & \text{ otherwise.} \end{cases}
\end{array}
\end{eqnarray}
\end{thm}

\begin{rem}
Note that these formulae are identical to the one (\ref{actonfundinf}) obtained for the fundamental $\U_q(\hat{sl}_\infty)$-modules. Actually this result generalizes \cite[Theorem 4.11]{mansuy_quantum_2012} for the extremal fundamental loop weight modules of $\U_q(sl_{n+1}^{tor})$. In fact one can rewrite the action on $V(Y_{\ell, a}Y_{0, aq^\ell}^{-1})$ in the context of monomial realizations: it is completely describe by the monomial $\U_q(sl_\infty)$-crystal $\mathcal{M}(Y_{\ell, a}Y_{0, aq^\ell}^{-1})$ and the formulae given in \cite[Theorem 4.11]{mansuy_quantum_2012}.
\end{rem}

\begin{proof}
We have shown that the $\U_q(\hat{sl}_\infty)$-module $V(Y_{\ell, a}Y_{0, aq^\ell}^{-1})$ is thin. So to determine the action on this module, it suffices to know the action of the horizontal quantum affine subalgebra $\U_q^h(\hat{sl}_\infty)$. The action of all the algebra $\U_q(\hat{sl}_\infty)$ can be deduced from it by (\ref{relxcartpl}) and (\ref{relxcartmo}).

So let us determine the action of $x_{i,0}^+$ ($i \in \ZZ$) on $V(Y_{\ell, a}Y_{0, aq^\ell}^{-1})$ (we proceed in the same way for $x_{i, 0}^{-}$). Let  $T = (\cdots < i_{\ell-2} < i_{\ell-1} < i_{\ell})$ and $T' = (j_1 < j_2 < j_3 < \cdots )$ be semi-standard tableaux such that $T \otimes T' \in \mathcal{T}_\ell$. Denote by $(T \otimes T')_a$ the vector $T_a \otimes T'_{aq^{\ell}} \in V(Y_{\ell, a}Y_{0, aq^\ell}^{-1})$. Then $x_{i, 0}^- \cdot (T \otimes T')_a$ is equal to
\begin{eqnarray*}
\begin{array}{c}
\sum_{s \geq 1} \delta_{\{j_s = i \}} \cdot T_a \otimes (\tilde{f}_i \cdot T')_{aq^{\ell}}
 + (\tilde{f}_i \cdot T)_a \otimes T'_{aq^{\ell}} \cdot \displaystyle \sum_{ r  \leq \ell, s \geq 1} \delta_{\{i_r = i \}}  \times  \\
 \left( 1 + \delta_{\{j_s = i + 1 \}} \delta_{\{j_{s-1} \neq i \}} \left( \psi(q^{2(r-s)+2})^{-1} - 1 \right) + \delta_{\{j_s = i \}} \delta_{\{j_{s+1} \neq i+1 \}} \left( \psi(q^{2(r-s)}) - 1 \right) \right) .
\end{array}
\end{eqnarray*}

\noindent We have to consider the following cases
\begin{itemize}
\item[-] $ i > \alpha_T$: then $i, i+1 \in T'$. So we get
\begin{eqnarray*}
x_{i, 0}^- \cdot (T \otimes T')_a = (\tilde{f}_i \cdot T)_a \otimes T'_{aq^{\ell}} = \left(\tilde{f}_i \cdot (T \otimes T') \right)_a.
\end{eqnarray*}

\item[-] $i < \alpha_T$: then we have $i, i+1 \in T$ and
\begin{eqnarray*}
x_{i, 0}^- \cdot (T \otimes T')_a = T_a \otimes (\tilde{f}_i \cdot T')_{aq^{\ell}} = \left( \tilde{f}_i \cdot (T \otimes T') \right)_a.
\end{eqnarray*}

\item[-] $i = \alpha_T$: then we have $i \in T, i+1 \notin T$ and $\beta_{T'} \leq i+1$.

Assume that $\beta_{T'} < i+1$. In this case, $i, i+1 \in T'$ and we get
\begin{eqnarray*}
x_{i, 0}^- \cdot (T \otimes T')_a = (\tilde{f}_i \cdot T)_a \otimes T'_{aq^{\ell}} = \left( \tilde{f}_i \cdot (T \otimes T') \right)_a.
\end{eqnarray*}

If $\beta_{T'} = i+1$, we have $i \notin T', \alpha_T = \beta_{T'}-1$ and
\begin{align*}
x_{i, 0}^- \cdot (T \otimes T')_a & = \psi(q^{2(\alpha_T - \beta_{T'}) + 2})^{-1} (\tilde{f}_i \cdot T)_a \otimes T'_{aq^{\ell}} = \psi(1)^{-1} (\tilde{f}_i \cdot T)_a \otimes T'_{aq^{\ell}} = 0 \\
& = \left( \tilde{f}_i \cdot (T \otimes T') \right)_a.
\end{align*}
\end{itemize}

So we have shown that $x_{i, 0}^- \cdot (T \otimes T')_a = \left( \tilde{f}_i \cdot (T \otimes T') \right)_a$. We proceed in the same way for $x_{i,0}^+$.
\end{proof}

As consequences of this theorem, we obtain

\begin{prop}\label{felwmirredinf}
Fix $\ell \geq 1$ and $a \in \CC^{\ast}$. Then $V(Y_{\ell, a}Y_{0, aq^\ell}^{-1})$ is irreducible as a $\U_q(\hat{sl}_\infty)$-module and as a $\U_q^h(\hat{sl}_\infty)$-module.
\end{prop}

\noindent To prove this proposition, we use the following result (which is an analogue of \cite[Lemma 4.8]{mansuy_quantum_2012}).

\begin{lem}\label{lirepcryinf}
Let $\mathcal{B}$ be a $\U_q(sl_\infty)$-crystal. Assume that $V$ is a $\U_q(sl_\infty)$-module with basis $(v_b)_{b \in \mathcal{B}}$ indexed by $\mathcal{B}$ which satisfies
\begin{eqnarray}\label{lienrepcryinf}
\wt(v_b) = \wt(b), \ (x_{i}^{+})^{(k)} \cdot v_b = v_{\tilde{e}_i^k \cdot b} \text{ and } (x_{i}^{-})^{(k)} \cdot v_b = v_{\tilde{f}_i^k \cdot b}
\end{eqnarray}
for all $b \in \mathcal{B}, i \in I$ and $k \in \NN$, where $v_0 = 0$ by convention. If the element $b \in \mathcal{B}$ is extremal of weight $\lambda$, then the vector $v_b$ is an extremal vector of weight $\lambda$. Furthermore if the crystal $\mathcal{B}$ is connected, then the $\U_q(sl_\infty)$-module $V$ is cyclic generated by any $v_b$ with $b \in \mathcal{B}$.
\end{lem}

\noindent The proof is similar to the one of \cite[Lemma 4.8]{mansuy_quantum_2012}. We recall it for the convenience of the reader.

\begin{proof}
Assume that $b \in \mathcal{B}$ is extremal of weight $\lambda$: there exists $\{b_w\}_{w \in W}$ such that $b_{Id}= b$ and
\begin{eqnarray}\label{eqinterinf}
\begin{array}{c}
\tilde{e}_i \cdot b_w=0 \text{ and }(\tilde{f}_i)^{w(\lambda)(h_i)} \cdot b_w= b_{s_i(w)} \text{ if } w(\lambda)(h_i)\geq 0,\\
\tilde{f}_i \cdot b_w=0 \text{ and }(\tilde{e}_i)^{-w(\lambda)(h_i)} \cdot b_w= b_{s_i(w)} \text{ if } w(\lambda)(h_i)\leq 0.
\end{array}
\end{eqnarray}
For all $w \in W$, set $v_w = v_{b_w}$. By (\ref{lienrepcryinf}) and (\ref{eqinterinf}), $\{v_w\}_{w \in W}$ satisfies $v_{Id} = v_{m}$ and
 $$x_i^{\pm} \cdot v_w=0 \text{ if } \pm w(\lambda)(h_i)\geq 0 \text{ and }(x_i^{\mp})^{(\pm w(\lambda)(h_i))} \cdot v_w=v_{s_i(w)}.$$
Hence the vector $v_m$ is extremal of weight $\lambda$.

Assume that the crystal $\mathcal{B}$ is connected and fix $b \in \mathcal{B}$. For $b' \in \mathcal{B}$, there exists a product $s$ of Kashiwara operators such that $s(b) = b'$. Consider the corresponding operator $S \in \U_q(sl_\infty)$ at the level of $V$, i.e. $S$ has the same expression as $s$ where the operators $\tilde{e}_i^k$ (resp. $\tilde{f}_i^k$) are replaced by $(x_{i}^{+})^{(k)}$ (resp. $(x_i^{-})^{(k)}$) in the product ($k \in \NN, i \in I$). By (\ref{lienrepcryinf}), $S(v_b) = v_{s(b)} = v_{b'}$ and the $\U_q(sl_\infty)$-module $V$ is cyclic generated by $v_b$.
\end{proof}

We are now able to prove Proposition \ref{felwmirredinf}.

\begin{proof}
Let us consider a non trivial sub-$\U_q^h(\hat{sl}_\infty)$-module $V$ of $V(Y_{\ell, a}Y_{0, aq^\ell}^{-1})$. By the $q$--character formula (\ref{qchainf}) all the weight spaces of $V(Y_{\ell, a}Y_{0, aq^\ell}^{-1})$ are of dimension one. So there exists $T \in \mathcal{T}_{[1, \ell]}$ such that $T_a \in V$. By Lemma \ref{lirepcryinf}, the vector $T_a$ generates $V(Y_{\ell, a}Y_{0, aq^\ell}^{-1})$ and $V = V(Y_{\ell, a}Y_{0, aq^\ell}^{-1})$. Hence $V(Y_{\ell, a}Y_{0, aq^\ell}^{-1})$ is simple as a $\U_q(\hat{sl}_\infty)$-module and as a $\U_q^h(\hat{sl}_\infty)$-module. 
\end{proof}

\begin{thm}\label{thmelwminf}
Fix $\ell \geq 1$ and $a \in \CC^{\ast}$. Then $V(Y_{\ell, a}Y_{0, aq^\ell}^{-1})$ is an extremal loop weight $\U_q(\hat{sl}_\infty)$-module generated by the vector $(1 < 2 < \cdots < \ell)_a$ of $\ell$-weight $Y_{\ell, a}Y_{0, aq^\ell}^{-1}$.
\end{thm}

\begin{proof}
By the fusion product construction, $V(Y_{\ell, a}Y_{0, aq^\ell}^{-1})$ is an integrable $\U_q(\hat{sl}_\infty)$-module. As it is also irreducible, $V(Y_{\ell, a}Y_{0, aq^\ell}^{-1})$ is generated by the vector $T =(1 < 2 < \cdots < \ell)_a$. The extremality of this vector for the horizontal quantum affine subalgebra follows from Lemma \ref{lirepcryinf} and the fact that $T$ is an extremal element in $\mathcal{T}_{[1, \ell]}$. So it remains to prove that for all $w \in V(Y_{\ell, a}Y_{0, aq^\ell}^{-1})$ and $J = [a, b] \subset \ZZ$ finite, $\U_q(\hat{sl}_\infty) \cdot w$ is finite-dimensional: this follows by Lemma \ref{tpalgv}.
\end{proof}

\begin{rem}
Set $n \in \NN^{\ast}, \ell \geq 1$ and $a \in \CC^{\ast}$. Let us consider the \linebreak $\U_q(\hat{sl}_\infty)_{[-n, \ell + n]}$-modules 
$$V(Y_{\ell, a})_{[-n, \ell + n]} = \U_q(\hat{sl}_\infty)_{[-n, \ell + n]} \cdot v^+ \text{ and } V(Y_{0, aq^{\ell}}^{-1})_{[-n, \ell + n]} = \U_q(\hat{sl}_\infty)_{[-n, \ell + n]} \cdot v^-$$
where $v^+$ and $v^-$ are $\ell$-highest weight vector and $\ell$-lowest weight vector of the \linebreak $\U_q(\hat{sl}_{\infty})$-modules $V(Y_{\ell, a})$ and $V(Y_{0, aq^{\ell}}^{-1})$ respectively. Using the result of this section, we deduce that $\Delta$ endows the tensor product $$V(Y_{\ell, a})_{[-n, \ell + n]} \otimes V(Y_{0, aq^{\ell}}^{-1})_{[-n, \ell + n]}$$ with a structure of $\U_q(\hat{sl}_\infty)_{[-n, \ell + n]}$-module. Let us set
$$V(Y_{\ell, a}Y_{0, aq^{\ell}}^{-1})_{[-n, \ell + n]} = \U_q(\hat{sl}_\infty)_{[-n, \ell + n]} \cdot v^+ \otimes v^- \subset V(Y_{\ell, a})_{[-n, \ell + n]} \otimes V(Y_{0, aq^{\ell}}^{-1})_{[-n, \ell + n]}.$$
We give some facts about the $\U_q(\hat{sl}_\infty)_{[-n, \ell + n]}$-module $V(Y_{\ell, a}Y_{0, aq^{\ell}}^{-1})_{[-n, \ell + n]}$ (consequences of the study done in this section): it is irreducible and finite-dimensional. In particular, this is a simple $\ell$-highest weight $\U_q(\hat{sl}_\infty)_{[-n, \ell + n]}$-module. Its $q$--character is
\begin{eqnarray*}
\chi_q(V(Y_{\ell, a}Y_{0, aq^{\ell}}^{-1})_{[-n, \ell + n]}) = \sum_{T \in (\mathcal{T}_\ell)_{[-n, \ell + n]}} m_T
\end{eqnarray*}
where $(\mathcal{T}_\ell)_{[-n, \ell + n]}$ is the finite set of semi-standard tableaux $$T = (-n \leq i_1 < i_2 < \cdots < i_{\ell} \leq \ell + n + 1).$$ The $\ell$-highest weight of $V(Y_{\ell, a}Y_{0, aq^{\ell}}^{-1})_{[-n, \ell + n]}$ corresponds to the dominant monomial in its $q$--character formula. It is obtained for $T = (-n < -n+1 < \cdots < -n+ \ell-1 )$. Hence we have
$$V(Y_{\ell, a}Y_{0, aq^{\ell}}^{-1})_{[-n, \ell + n]} = V(m_T)_{[-n, \ell + n]} = V(Y_{-n+ \ell -1, aq^{-n-1}})_{[-n, \ell + n]}.$$
Then $V(Y_{\ell, a}Y_{0, aq^{\ell}}^{-1})_{[-n, \ell + n]}$ is the fundamental $\U_q(\hat{sl}_\infty)_{[-n, \ell + n]}$-module of $\ell$-highest weight $Y_{-n+ \ell -1, aq^{-n-1}}$.
\end{rem}

\section{Fusion product of extremal fundamental loop weight modules}

In this section we study fusion products of extremal loop weight modules $V(Y_{\ell, a}Y_{0, a q^\ell}^{-1})$ with $\ell \geq 1$ and $a \in \CC^{\ast}$. We obtain new families of extremal loop weight modules with basis labelled by semi-standard tableaux.

In the first part, we determine existence conditions of $\U_q(\hat{sl}_\infty)$-module structure on the tensor product of extremal fundamental loop weight modules $V(Y_{\ell, a}Y_{0, a q^\ell}^{-1})$. In the second part we study the case of fusion products of $k$ extremal fundamental loop weight modules $V(Y_{\ell, a_i}Y_{0, a_i q^\ell}^{-1})$ with generic non-zero complex parameters $a_1, a_2, \cdots , a_k \in \CC^{\ast}$ (Theorem \ref{thmgencaseinf}): it is an extremal loop weight module of $\ell$-weight
$$Y_{\ell, a_1}\cdots Y_{\ell, a_k}Y_{0, a_1q^\ell}^{-1} \cdots Y_{0, a_k q^\ell}^{-1}.$$
In the third part we treat the case of fusion products of extremal fundamental loop weight modules when parameters are chosen non-generic. More precisely we consider the fusion product of vector representations $V(Y_{1,a}Y_{0, aq}^{-1})$ when the set of parameters forms a $q$-segment $\{a, aq^{-2}, \cdots, aq^{-2(k-1)}\}$ ($k \in \NN^{\ast}$, $ a \in \CC^{\ast}$). We recover in that way all the extremal fundamental loop weight modules (Theorem \ref{thmtpvrisomeflwm}). Furthermore we obtain new extremal loop weight modules of $\ell$-weight
$$Y_{\ell, aq^{-\ell+1}}Y_{\ell, aq^{-\ell+3}} \cdots Y_{\ell, aq^{-\ell+1+2k}} Y_{0,aq}^{-1}Y_{0,aq^3}^{-1} \cdots Y_{0,aq^{1+2k}}^{-1} \text{ (Theorem \ref{thmtpelwminf})}.$$

\subsection{Existence conditions}

Set $\ell \geq 1$. Let us consider the tensor product
$$V(Y_{\ell, a}Y_{0, a q^\ell}^{-1}) \otimes V(Y_{\ell, b}Y_{0, b q^\ell}^{-1}) \text{ with } a, b \in \CC^{\ast}.$$
We determine when the action of $\U_q(\hat{sl}_\infty)$ is well-defined on it.

\begin{prop}\label{proptpdefinf}
\begin{itemize}
\item[(i)] The action of $\U_q(\hat{sl}_\infty)$ is well-defined on the tensor product $V(Y_{\ell, a}Y_{0, a q^\ell}^{-1}) \otimes V(Y_{\ell, b}Y_{0, b q^\ell}^{-1})$ if and only if $$\dfrac{a}{b} \notin \{q^{-2\ell+2}, q^{-2\ell+4}, \cdots, q^{2\ell-2} \}.$$
\item[(ii)] Assume that $\dfrac{a}{b} = q^{- 2\ell}$. The module $V(Y_{\ell, a}Y_{0, a q^\ell}^{-1}) \otimes V(Y_{\ell, b}Y_{0, b q^\ell}^{-1})$ has a submodule spanned by vectors of the form 
$$T_a \otimes T'_b \text{ with } i_1 \leq j_\ell$$
where $T = (i_1 < \cdots < i_\ell)$ and $T' = (j_1 < \cdots < j_\ell)$. The submodule and the quotient module are irreducible and thin.
\item[(iii)] Assume that $\dfrac{a}{b} = q^{2\ell}$. The module $V(Y_{\ell, a}Y_{0, a q^\ell}^{-1}) \otimes V(Y_{\ell, b}Y_{0, b q^\ell}^{-1})$ has a submodule spanned by vectors of the form 
$$T_a \otimes T'_b \text{ with } i_\ell < j_1$$
where $T = (i_1 < \cdots < i_\ell)$ and $T' = (j_1 < \cdots < j_\ell)$. The submodule and the quotient module are irreducible and thin.
\item[(iv)] In all the other cases, $V(Y_{\ell, a}Y_{0, a q^\ell}^{-1}) \otimes V(Y_{\ell, b}Y_{0, b q^\ell}^{-1})$ is a thin, irreducible \linebreak $\U_q(\hat{sl}_\infty)$-module.
\end{itemize}
\end{prop}

\begin{rem}
This result generalizes the one for the (specialized) vector representations of $\U_q(sl_{n+1}^{tor})$ $(n \geq 2)$ given in \cite{feigin_representations_2013, mansuy_extremal_2013}.
\end{rem}

\begin{proof}
Let $T, T' \in \mathcal{T}_{[1, \ell]}$ and $a, b \in \CC^{\ast}$. We determine the action of $x_i^-(z)$ on \linebreak $T_a \otimes T'_b \in V(Y_{\ell, a}Y_{0, a q^\ell}^{-1}) \otimes V(Y_{\ell, b}Y_{0, b q^\ell}^{-1})$: it is equal to
\begin{eqnarray*}
\begin{array}{c}
\displaystyle \sum_{1 \leq s \leq \ell} \delta_{\{j_s = i \}} \delta(bq^{i + \ell + 1 - 2s}) \cdot T_a \otimes (\tilde{f}_i \cdot T')_b + \displaystyle \sum_{1 \leq r, s \leq \ell} \delta_{\{i_r = i \}} \delta(aq^{i + \ell + 1 - 2r}) \\
\times  \left( 1 + \delta_{\{j_s = i + 1 \}} \delta_{\{j_{s-1} \neq i \}} \left( \psi(\frac{b}{a} q^{2(r-s)+2})^{-1} - 1 \right) + \delta_{\{j_s = i \}} \delta_{\{j_{s+1} \neq i+1 \}} \left( \psi(\frac{b}{a} q^{2(r-s)}) - 1 \right) \right) \\
\times (\tilde{f}_i \cdot T)_a \otimes T'_b.
\end{array}
\end{eqnarray*}
So the action of the operators $x_i^-(z)$ is well-defined on $ V(Y_{\ell, a}Y_{0, a q^\ell}^{-1}) \otimes V(Y_{\ell, b}Y_{0, b q^\ell}^{-1})$ if and only if
$$\dfrac{a}{b} \neq q^{2(r-s)} \text{ for all } 1 \leq r, s \leq \ell$$
or in an equivalent way, $\dfrac{a}{b} \notin \{q^{-2\ell+2}, q^{-2\ell+4}, \cdots, q^{2\ell-2} \}$. We proceed in the same way for the $x_i^+(z)$. We obtain the first assertion by Lemma \ref{lemsubminf}.

Now assume that $\dfrac{a}{b} = q^{-2\ell}$. Let us consider the subvector space $V$ of $$ V(Y_{\ell, a}Y_{0, a q^\ell}^{-1}) \otimes V(Y_{\ell, b}Y_{0, b q^\ell}^{-1})$$ spanned by vectors of the form 
$$T_a \otimes T'_b \text{ with } i_1 \leq j_\ell$$
where $T = (i_1 < \cdots < i_\ell)$ and $T' = (j_1 < \cdots < j_\ell)$. Let $i \in \ZZ$ and $T = (i < i_2 < \cdots < i_{\ell})$, $T' = (j_1 < \cdots < j_{\ell-1} < i)$ be semi-standard tableaux. We have
\begin{align*}
x_i^-(z) \cdot T_a \otimes T'_b &= \delta(aq^{i + \ell + 1}z) T_a \otimes (\tilde{f}_i \cdot T')_b + \delta(aq^{i+\ell-1}z) \psi(q^{2}) (\tilde{f}_i \cdot T)_a \otimes T'_b\\
& = \delta(aq^{i + \ell + 1}) T_a \otimes (\tilde{f}_i \cdot T')_b,\\
x_{i-1}^+(z) \cdot T_a \otimes T'_b &= \delta(aq^{i+\ell-1}z) (\tilde{e}_{i-1} \cdot T)_a \otimes T'_{b} + \delta(aq^{i+\ell+1}z) \psi (1)^{-1} T_a \otimes (\tilde{e}_{i-1} \cdot T')_{aq^{2\ell}}\\
&= \delta(aq^{i+\ell-1}z) (\tilde{e}_{i-1} \cdot T)_a \otimes T'_{b}.
\end{align*}
By these computations, $V$ is a sub-$\U_q(\hat{sl}_\infty)$-module of $ V(Y_{\ell, a}Y_{0, a q^\ell}^{-1}) \otimes V(Y_{\ell, b}Y_{0, b q^\ell}^{-1})$. We proceed in the same way for the third point. The thin property and the irreducibility of the modules in $(ii)$, $(iii)$ and $(iv)$ follow by straightforward computations.
\end{proof}

\subsection{The generic case}

Let us give a first result about the fusion product of extremal fundamental loop weight modules when the non-zero complex parameters are chosen generics.

\begin{thm}\label{thmgencaseinf}
Let $\ell \geq 1$, $k \in \NN^{\ast}$ and $a_1, \cdots, a_k \in \CC^{\ast}$ be such that $\frac{a_i}{a_j} \neq 1$ and $\frac{a_i}{a_j} \neq q^{\pm 2}$ for all $i < j$. Then the fusion product of extremal fundamental loop weight modules $$V = V(Y_{\ell, a_1}Y_{0, a_1 q^\ell}^{-1}) \otimes V(Y_{\ell, a_2}Y_{0, a_2 q^\ell}^{-1}) \otimes \cdots \otimes V(Y_{\ell, a_k}Y_{0, a_k q^\ell}^{-1})$$ is an irreducible extremal loop weight module of $\U_q(\hat{sl}_\infty)$ of $\ell$-weight $$Y_{\ell, a_1}\cdots Y_{\ell, a_k}Y_{0, a_1q^\ell}^{-1} \cdots Y_{0, a_k q^\ell}^{-1}.$$
\end{thm}

\begin{proof}
As a direct consequence of Proposition \ref{proptpdefinf}, the $\U_q(\hat{sl}_\infty)$-module $V$ is irreducible. Furthermore as $V(Y_{\ell, a_j}Y_{0, a_j q^\ell})$ are extremal loop weight modules for all $1 \leq j \leq k$, there fusion product $V$ satisfies the third condition of Definition \ref{defelminf}. Eventually for $T \in \mathcal{T}_{[1, \ell]}$, set $T' = \tilde{f}_i \cdot T$ and $T'' = \tilde{e}_i \cdot T$. We have for all $i \in \ZZ$ (see the proof of \cite[Theorem 5.3]{mansuy_extremal_2013})
$$(x_{i,0}^{-})^{(k)} \cdot T_{a_1} \otimes \cdots \otimes T_{a_k} = T'_{a_1} \otimes \cdots \otimes T'_{a_k}$$
and
$$(x_{i,0}^{+})^{(k)} \cdot T_{a_1} \otimes \cdots \otimes T_{a_k} = T''_{a_1} \otimes \cdots \otimes T''_{a_k}.$$
The extremality of $(1 < \cdots < \ell)_{a_1} \otimes \cdots \otimes (1 < \cdots < \ell)_{a_k}$ follows.
\end{proof}

\subsection{The non-generic case}

Let us consider the extremal fundamental loop weight module $V(Y_{1,a}Y_{0,aq}^{-1})$ ($a \in \CC^{\ast}$) of $\U_q(\hat{sl}_\infty)$. This module has a basis labelled by the tableaux
$$T_a = \ffbox{j}_a \text{ with } j \in \ZZ.$$
Furthermore the action of $\U_q(\hat{sl}_\infty)$ is known, given by
\begin{eqnarray*}
x_{i}^{+}(z) \cdot \ffbox{j}_a &=& \delta_{i, j-1} \delta(a q^{j-1}z) \cdot \ffbox{j-1}_a,\\
x_{i}^{-}(z) \cdot \ffbox{j}_a &=& \delta_{i, j} \delta(a q^{j}z) \cdot \ffbox{j+1}_a,\\
\phi_{i}^{\pm}(z) \cdot \ffbox{j}_a &=& \begin{cases} \psi(aq^{j}z) \cdot \ffbox{j}_a & \text{ if } i = j,\\
\psi(aq^{j+1}z)^{-1}\cdot \ffbox{j}_a & \text{ if } i = j-1,\\
\ffbox{j}_a & \text{ otherwise.}
\end{cases}
\end{eqnarray*}

We denote this module $V(\ffbox{1}_a)$ in the following. It is defined in \cite{feigin_representations_2013, mansuy_quantum_2012, mansuy_extremal_2013} with different methods for the quantum toroidal algebra $\U_q(sl_{n+1}^{tor})$ ($n\geq 2$) and called (specialized) \textit{vector representation}.

For all $k \in \NN^{\ast}$ and $ a \in \CC^{\ast}$, we consider the tensor product of vector representations
$$V(\ffbox{1}_a) \otimes V(\ffbox{1}_{aq^{-2}}) \otimes \cdots \otimes V(\ffbox{1}_{aq^{-2(k-1)}}).$$
By Proposition \ref{proptpdefinf}, the coproduct $\Delta$ endows this tensor product with a structure of $\U_q(\hat{sl}_\infty)$-module. Let us give a first result about it (the proof is exactly the same as \cite[Theorem 4.5]{mansuy_extremal_2013} and is not recalled).

\begin{thm}\label{elwmtpinf}
The $\U_q(\hat{sl}_\infty)$-module $V(\ffbox{1}_a) \otimes \cdots \otimes V(\ffbox{1}_{aq^{-2(k-1)}})$ is an extremal loop weight module of $\ell$-weight $$Y_{1,a}Y_{1,aq^{-2}} \cdots Y_{1, aq^{-2(k-1)}}Y_{0,aq}^{-1}Y_{0,aq^{-1}}^{-1} \cdots Y_{0, aq^{-2(k-1)+1}}^{-1}.$$
\end{thm}

Denote by $V \left(\begin{array}{c} \ffbox{1}_{\ } \\ \vdots_{\ } \\ \ffbox{k}_{a} \end{array} \right)$ the subvector space of $V(\ffbox{1}_a) \otimes \cdots \otimes V(\ffbox{1}_{aq^{-2(k-1)}})$ generated by $$\ffbox{i_1}_a \otimes \ffbox{i_2}_{aq^{-2}} \otimes \cdots \otimes \ffbox{i_k}_{aq^{-2(k-1)}} \text{ with } i_1 < i_2 < \cdots < i_k.$$

\begin{prop}
The coproduct $\Delta$ endows the vector space $V \left(\begin{array}{c} \ffbox{1}_{\ } \\ \vdots_{\ } \\ \ffbox{k}_{a} \end{array} \right)$ with a structure of thin $\U_q(\hat{sl}_\infty)$-module.
\end{prop}

\begin{proof}
The action of the $x_i^{\pm}(z)$ is well-defined on it by Proposition \ref{proptpdefinf}. It remains to show that for a vector $$v = \ffbox{i_1}_a \otimes \ffbox{i_2}_{aq^{-2}} \otimes \cdots \otimes \ffbox{i_k}_{aq^{-2(k-1)}} \text{ with } i_1 < i_2 < \cdots < i_k,$$
$x_i^{\pm}(z) \cdot v \in V$ for all $i \in \ZZ$. This is a consequence of the equalities
\begin{eqnarray}\label{eqinterinf2}
\begin{array}{ccc}
x_{i}^{+}(z) \cdot \ffbox{i}_{a} \otimes \ffbox{i+1}_{aq^{-2}} &=& \delta(a q^{i-2}z) \psi(q^{2}) \cdot \ffbox{i}_{a} \otimes \ffbox{i}_{aq^{-2}} = 0,\\
x_{i}^{-}(z) \cdot \ffbox{i}_{a} \otimes \ffbox{i+1}_{aq^{-2}} &=& \delta(a q^{i}z) \psi(1)^{-1} \cdot \ffbox{i+1}_{a} \otimes \ffbox{i+1}_{aq^{-2}} = 0.
\end{array}
\end{eqnarray}

\noindent Eventually these modules are thin, the weight of the vector $\ffbox{i_1}_a \otimes \cdots \otimes \ffbox{i_k}_{aq^{-2(k-1)}}$ being completely determined by the sequence $i_1 < i_2 < \cdots < i_k$.
\end{proof}

\begin{thm}\label{thmtpvrisomeflwm}
The $\U_q(\hat{sl}_\infty)$-module $V = V \left(\begin{array}{c} \ffbox{1} \\ \vdots \\ \ffbox{\ell} \end{array}_a \right)$ ($\ell \geq 1, a \in \CC^{\ast}$) is isomorphic to the extremal fundamental loop weight module $V(Y_{\ell, aq^{-\ell+1}}Y_{0,aq}^{-1})$.
\end{thm}

\begin{proof}
Let us define the morphism of vector spaces $f : V(Y_{\ell, aq^{-\ell+1}}Y_{0,aq}) \rightarrow V$ by setting
$$f(T) = \ffbox{i_1}_a \otimes \ffbox{i_2} \otimes \cdots \otimes \ffbox{i_\ell}_{aq^{-2(\ell-1)}}$$
for all semi-standard tableaux $T = (i_1 < i_2 < \cdots < i_\ell) $ in $\mathcal{T}_{[1, \ell]}$. Hence defined $f$ is an isomorphism of $\U_q(\hat{sl}_\infty)$-modules: this follows by straightforward computations.
\end{proof}

\begin{prop}\label{proptpeflwminf}
Fix $a \in \CC^{\ast}, k \in \NN^{\ast}$ and $\ell_1, \ell_2, \cdots, \ell_k \geq 1$. Then $\Delta$ endows the tensor product
$$V \left(\begin{array}{c} \ffbox{1} \\ \vdots \\ \ffbox{\ell_1} \end{array}_{a} \right) \otimes V \left(\begin{array}{c} \ffbox{1} \\ \vdots \\ \ffbox{\ell_2} \end{array}_{aq^{-2\ell_1}} \right) \otimes \cdots \otimes V \left(\begin{array}{c} \ffbox{1} \\ \vdots \\ \ffbox{\ell_k} \end{array}_{aq^{-2(\ell_1 + \cdots + \ell_{k-1})}} \right)$$
with a structure of $\U_q(\hat{sl}_\infty)$-module. Furthermore, it has a submodule isomorphic to
$$V \left(\begin{array}{c} \ffbox{1} \\ \vdots \\ \ffbox{\ell} \end{array}_{a} \right) \text{ with } \ell = \ell_1 + \ell_2 + \cdots + \ell_k.$$
\end{prop}

\begin{proof}
By Proposition \ref{proptpdefinf}, the action on the tensor product is well-defined. Furthermore let us consider the subvector space generated by vectors of the form
$$\ffbox{i_1^{(1)}} \otimes \cdots \otimes \ffbox{i_{\ell_1}^{(1)}} \otimes \ffbox{i_1^{(2)}} \otimes \cdots \otimes \ffbox{i_{\ell_2}^{(2)}} \otimes \cdots \otimes \ffbox{i_1^{(k)}} \otimes \cdots \otimes \ffbox{i_{\ell_k}^{(k)}}$$
with $i_1^{(1)} < \cdots < i_{\ell_1}^{(1)} < i_1^{(2)} < \cdots < i_{\ell_2}^{(2)} < \cdots < i_1^{(k)} < \cdots < i_{\ell_k}^{(k)}$. It is by definition the submodule $V \left(\begin{array}{c} \ffbox{1} \\ \vdots \\ \ffbox{\ell} \end{array}_{a} \right)$ of $V(\ffbox{1}_a) \otimes \cdots \otimes V(\ffbox{1}_{aq^{-2(\ell-1)}})$ with $\ell = \ell_1 + \cdots + \ell_k$.
\end{proof}

Set $\ell, k \geq 1$ and $a \in \CC^{\ast}$. Let us consider the tensor product $V$ of vector representations
\begin{align*}
V(\ffbox{1}_a) \otimes \cdots \otimes V(\ffbox{1}_{aq^{-2(\ell-1)}}) \otimes \cdots \otimes V(\ffbox{1}_{aq^{-2(1-k)}}) \otimes \cdots \otimes V(\ffbox{1}_{aq^{-2(\ell-k)}}).
\end{align*}

Let $\mathcal{T}_{[1, \ell] \times [1, k]}$ be the set of semi-standard tableaux $T = (T_{i,j})_{(i, j) \in [1, \ell] \times [1, k]}$. Then $(\mathcal{T}_{[1, \ell] \times [1, k]}, \wt, \tilde{e}_i, \tilde{f}_i)$ is a $\U_q(sl_\infty)$-crystal \cite{kashiwara_bases_2002}. To a tableaux $T = (T_{i,j}) \in \mathcal{T}_{[1, \ell] \times [1, k]}$, it corresponds a vector $T_a$ in $V$, defined by (we read the tableaux from top to bottom and from left to right)
$$T_a = \ffbox{T_{1,1}}_a \otimes \cdots \otimes \ffbox{T_{\ell,1}}_{aq^{-2(\ell-1)}} \otimes \cdots \otimes \ffbox{T_{1,k}}_{aq^{-2(1-k)}} \otimes \cdots \otimes \ffbox{T_{\ell,k}}_{aq^{-2(\ell-k)}}.$$
Denote by $T^e = (T_{i,j})$ the semi-standard tableaux such that $T_{i, j} = i$ for all \linebreak $(i, j) \in [1, \ell] \times [1, k]$:
$$T^e = \begin{tiny}
\begin{tabular}{|c|c|c|}
\hline 
1 & \ldots & 1 \\ 
\hline 
\vdots &   & \vdots \\ 
\hline 
$\ell$& \ldots  & $\ell$ \\
\hline 
\end{tabular}
\end{tiny}.$$
Let $V(T_a^e)$ be the subvector space of $V$ generated by vectors $T_a$ with $T \in \mathcal{T}_{[1, \ell] \times [1, k]}$.

\begin{prop}\label{propactwdtpelvmwmultinf}
Set $\ell, k \geq 1$ and $a \in \CC^{\ast}$. The coproduct $\Delta$ endows $V(T_a^e)$ with a structure of thin $\U_q(\hat{sl}_\infty)$-module.
\end{prop}

\begin{proof}
Let $T = (T_{i,j}) \in \mathcal{T}_{[1, \ell] \times [1, k]}$ be a semi-standard tableaux. We have to show that for all $s \in \ZZ$, $x_{s}^-(z) \cdot T_a$ is also in $V(T_a^e)$. Actually it suffices to consider the following cases:
\begin{itemize}
\item[-] if $T$ is such that there exist $1 \leq i \leq \ell-1$ and $1 \leq j \leq k$ such that $$T_{i,j}=s \text{ and } T_{i, j+1} = s+1.$$ Then by (\ref{eqinterinf2}), $x_{s}^-(z) \cdot T_a$ is in $ V(T_a^e)$.
\item[-] if $T$ is such that there exist $1 \leq i \leq \ell$ and $p \geq 1$, $1 \leq j \leq k-p$ such that
$$T_{i,j} = T_{i, j+1} = \cdots = T_{i, j+p} = s \text{ and } T_{i, j+p+1} \geq s+1.$$
Then the fact that $x_{s}^-(z) \cdot T_a$ belongs to $ V(T_a^e)$ is a consequence of the following equality in $V(\ffbox{1}_a) \otimes \cdots \otimes V(\ffbox{1}_{aq^{2p}})$
\begin{align*}
x_s^-(z) \cdot \ffbox{s}_a \otimes \cdots \otimes \ffbox{s}_{aq^{2p}} = \delta( a q^{s+2p} z ) \ffbox{s}_a \otimes \cdots \otimes \ffbox{s}_{aq^{2(p-1)}} \otimes \ffbox{s+1}_{aq^{2p}}.
\end{align*}
\end{itemize}
We proceed in the same way for the operators $x_s^+(z)$.
\end{proof}

\begin{thm}\label{thmtpelwminf}
Set $\ell, k \geq 1$ and $a \in \CC^{\ast}$. Then $V(T_a^e)$ is an irreducible extremal loop weight $\U_q(\hat{sl}_\infty)$-module of $\ell$-weight
$$m_{T^e} = Y_{\ell, aq^{-\ell+1}}Y_{\ell, aq^{-\ell+3}} \cdots Y_{\ell, aq^{-\ell+1+2(k-1)}} Y_{0,aq}^{-1}Y_{0,aq^3}^{-1} \cdots Y_{0,aq^{1+2(k-1)}}^{-1}.$$
\end{thm}

\begin{proof}
Let us recall that the semi-standard tableaux
$$T^e = \begin{tiny}
\begin{tabular}{|c|c|c|}
\hline 
1 & \ldots & 1 \\ 
\hline 
\vdots &   & \vdots \\ 
\hline 
$\ell$& \ldots  & $\ell$ \\
\hline 
\end{tabular}
\end{tiny}$$
is an extremal element of the $\U_q(sl_\infty)$-crystal $\mathcal{T}_{[1, \ell] \times [1, k]}$ of weight $k \Lambda_\ell - k \Lambda_0$. Furthermore we have
$$W \cdot T^e = \lbrace T = (T_{i,j})_{(i,j) \in [1, \ell] \times [1, k]} \vert T_{i,1} = T_{i, 2} = \cdots = T_{i, k} \text{ for all } 1 \leq i \leq \ell \rbrace.$$
Using the equalities ($s \in \ZZ$)
\begin{align*}
(x_{s, 0}^-)^{(k)} \cdot \ffbox{s}_a \otimes \cdots \otimes \ffbox{s}_{aq^{2(k-1)}} & = \ffbox{s+1}_a \otimes \cdots \otimes \ffbox{s+1}_{aq^{2(k-1)}},\\
(x_{s-1, 0}^+)^{(k)} \cdot \ffbox{s}_a \otimes \cdots \otimes \ffbox{s}_{aq^{2(k-1)}} & = \ffbox{s-1}_a \otimes \cdots \otimes \ffbox{s-1}_{aq^{2(k-1)}},
\end{align*}
we get
$$x_i^{\pm} \cdot T_a = 0 \text{ and } (x_i^{\mp})^{(\pm \wt(T)(h_i))} \cdot T_a = S_i(T)_a \text{ if } \pm \wt(T)(h_i) \geq 0$$
for all $T \in W \cdot T^e$. Then the result is a consequence of the following Lemma. 
\end{proof}

\begin{lem}
Let $V$ be a $\U_q(sl_\infty)$-module with basis $(v_b)_{b \in \mathcal{B}}$ indexed by a \linebreak $\U_q(sl_\infty)$-crystal $\mathcal{B}$. Assume that $b^e \in \mathcal{B}$ is extremal of weight $\lambda \in P$ and for all $i \in \ZZ$ and $b \in W \cdot b^e$,
$$\wt(v_b) = \wt(b), \ x_i^{\pm} \cdot v_b = 0 \text{ and } (x_i^{\mp})^{(\pm \wt(b)(h_i))} \cdot v_b = v_{S_i(b)} \text{ if } \pm \wt(b)(h_i) \geq 0.$$
Then $v_{b^e}$ is an extremal vector of weight $\lambda$.
\end{lem}

\begin{proof}
The proof is analogue to the one of Lemma \ref{lirepcryinf}.
\end{proof}

\section{Application to quantum toroidal algebras}

The quantum toroidal algebras $\U_q(sl_{n+1}^{tor})$ ($n \geq 2$) were introduced by Ginzburg-Kapranov-Vasserot in type A \cite{ginzburg_langlands_1995}. They are quantum groups analogs of elliptic Cherednik algebras \cite{cherednik_double_1995} to whom they are related via Schur-Weyl duality \cite{varagnolo_schur_1996}. The author has defined extremal loop weight modules in the context of $\U_q(sl_{n+1}^{tor})$ \cite{mansuy_quantum_2012, mansuy_extremal_2013}. The main motivation is the construction of finite-dimensional representations of the quantum toroidal algebras at roots of unity.

A combinatorial link between the representation theory of $\U_q(\hat{sl}_\infty)$ and of quantum toroidal algebras is conjectured in \cite{hernandez_algebra_2011}. This conjecture is proved for the class of Kirillov-Reshetikhin modules. The main motivation in \cite{hernandez_algebra_2011} is to predict $q$--character formulae for representations of $\U_q(sl_{n+1}^{tor})$.

In this section, we prove \cite[Conjecture 5.3]{hernandez_algebra_2011} for the particular family of extremal fundamental loop weight modules of $\U_q(\hat{sl}_\infty)$: by the combinatorial link highlighted in \cite{hernandez_algebra_2011}, they are related to $\U_q(sl_{n+1}^{tor})$-modules defined in \cite{mansuy_extremal_2013}. The aim is to construct extremal loop weight modules of $\U_q(sl_{n+1}^{tor})$. We show that we recover the extremal fundamental loop weight modules constructed in \cite{mansuy_quantum_2012, mansuy_extremal_2013}.

In the first part, we recall some definitions about the quantum toroidal algebras $\U_q(sl_{n+1}^{tor})$ ($n \geq 2$) and its combinatorial link with $\U_q(\hat{sl}_\infty)$ conjectured in \cite{hernandez_algebra_2011}. In the second part, we prove this conjecture for the class of extremal fundamental loop weight modules of $\U_q(\hat{sl}_\infty)$ (Theorem \ref{thmqchaactreptorinf}). We recover in that way the extremal fundamental loop weight modules of $\U_q(sl_{n+1}^{tor})$ constructed in \cite{mansuy_quantum_2012, mansuy_extremal_2013} (Theorem \ref{thmelwmtorinf}).

\subsection{Reminder about quantum toroidal algebras}

Let us recall the definition of the quantum toroidal algebra $\U_q(sl_{n+1}^{tor})$ ($n \geq 2$). Set $I_n = \ZZ / (n+1) \ZZ$. Let $C = (C_{\overline{i},\overline{j}})_{\overline{i}, \overline{j} \in I_n}$ be the Cartan matrix of type $A_n^{(1)}$,
$$ C_{\overline{i},\overline{i}} = 2 \text{ , } C_{\overline{i},\overline{i+1}} = C_{\overline{i+1}, \overline{i}} = -1 \text{ and } C_{\overline{i},\overline{j}} = 0 \text{ if } \overline{j} \neq \overline{i}, \overline{i \pm 1}.$$
The algebra $\U_q(sl_{n+1}^{tor})$ is defined by the same generators and relations as in Definition \ref{defqaainf} with $\overline{i}, \overline{j} \in I_n$.

The representation theory of $\U_q(sl_{n+1}^{tor})$ is similar to the one of $\U_q(\hat{sl}_\infty)$: the simple integrable $\ell$-highest weight modules are parametrized by Drinfeld polynomials. In particular, the Kirillov-Reshetikhin modules can be defined in an analogue way. One defines $q$--characters $\chi_{q, n}(V)$ as above for integrable representations $V$ with finite-dimensional $\ell$-weight spaces: it is a sum of elements in $A_n$, where $A_n$ is the group of monomials of the form
$$m = \prod_{\overline{i} \in I_n, a \in \CC^{\ast}} Y_{\overline{i}, a}^{u_{\overline{i}, a}(m)}, \ u_{\overline{i}, a}(m) \in \ZZ.$$

Consider the ring morphism
$$\phi_n : \ZZ[Y_{i,a}^{\pm 1}]_{i \in \ZZ, a \in \CC^{\ast}} \rightarrow \ZZ[Y_{\overline{i},a}^{\pm 1}]_{\overline{i} \in I_n, a \in \CC^{\ast}}$$
defined, for $i \in \ZZ$ and $a \in \CC^{\ast}$, by
$$\phi_n(Y_{i,a}^{\pm 1}) = Y_{\overline{i}, a}^{\pm 1}.$$

Let $\mathrm{Im}(\chi_q)$ (resp. $\mathrm{Im}(\chi_{q, n})$) be the image of the Grothendieck group corresponding to the category of integrable $\U_q(\hat{sl}_\infty)$-modules (resp. integrable $\U_q(sl_{n+1}^{tor})$-modules) belonging to $\mathcal{O}$. The morphism $\phi_n$ gives rise naturally to a group morphism
$$\phi_n : \mathrm{Im}(\chi_q) \rightarrow \ZZ[[Y_{\overline{i},a}^{\pm 1}]]_{\overline{i} \in I_n, a \in \CC^{\ast}}.$$
More precisely, one can show that $\phi_n \left(\mathrm{Im}(\chi_q) \right) \subset \mathrm{Im}(\chi_{q, n})$ by using the characterization stated in the proof of \cite[Theorem 4.2]{hernandez_algebra_2011}. Hence for $V$ a $\U_q(\hat{sl}_\infty)$-module in the category $\mathcal{O}_{\mathrm{int}}$, $\phi_n \left(\chi_q(V) \right)$ is the $q$--character of a representation of $\U_q(sl_{n+1}^{tor})$, which is a priori virtual. It is conjectured in \cite{hernandez_algebra_2011} that

\begin{conj}\label{conjherninf}
Let $V$ be a simple representation in $\mathcal{O}_{\mathrm{int}}$ for $\U_q(\hat{sl}_\infty)$. Then $\phi_n(\chi_q(V))$ is the $q$--character of an actual representation of $\U_q(sl_{n+1}^{tor})$.
\end{conj}

This conjectural link between the representation theory of $\U_q(\hat{sl}_\infty)$ and the one of quantum toroidal algebras is proved in \cite{hernandez_algebra_2011} for Kirillov-Reshetikhin modules. In the following section, we give one of the main results of the paper: we prove Conjecture \ref{conjherninf} in the context of extremal fundamental loop weight modules of $\U_q(\hat{sl}_\infty)$.

\subsection{Extremal loop weight modules for quantum toroidal algebras}

Let us consider the extremal fundamental loop weight module $V(Y_{\ell, a} Y_{0, aq^{\ell}}^{-1})$ with $\ell \geq 1$ and $a \in \CC^{\ast}$. Recall that we have determined its $q$--character: it is given by
\begin{equation*}
\chi_q(V(Y_{\ell, a}Y_{0, aq^\ell}^{-1})) = \sum_{T \in \mathcal{T}_{[1, \ell]}} m_T
\end{equation*}
where $\mathcal{T}_{[1, \ell]}$ is the set of semi-standard tableaux of shape $(\ell)$ and $$m_T = \prod_{1 \leq j \leq \ell} \ffbox{i_j}_{aq^{\ell+1-2j}}.$$ Furthermore, we have shown that
$$\{m_T \vert T \in \mathcal{T}_{[1, \ell]} \} = \mathcal{M}(Y_{\ell, a}Y_{0, aq^\ell}^{-1}).$$

Let us consider the ring morphism $\phi_n : \ZZ[Y_{i,a}^{\pm 1}]_{i \in \ZZ, a \in \CC^{\ast}} \rightarrow \ZZ[Y_{\overline{i},a}^{\pm 1}]_{\overline{i} \in I_n, a \in \CC^{\ast}}$. It cannot be extend to an application $\ZZ^{A} \rightarrow \ZZ^{A_n}$. However, $\phi_n$ gives rise to a well-defined application
$$\phi_n : \ZZ^{\mathcal{M}(Y_{\ell, a}Y_{0, aq^\ell}^{-1})} \longrightarrow \ZZ^{A_n}.$$
Actually by definition of the monomial crystal $\mathcal{M}$ (see \cite{kashiwara_realizations_2003, mansuy_quantum_2012, nakajima_$t$-analogs_2003}), $\phi_n$ can be defined in all its connected components.

\begin{prop}
$\phi_n \left( \chi_q(V(Y_{\ell, a}Y_{0, aq^\ell}^{-1})) \right)$ is the $q$--character of a virtual representation of $\U_q(sl_{n+1}^{tor})$.
\end{prop}

\begin{proof}
We have to check that $\phi_n \left( \chi_q(V(Y_{\ell, a}Y_{0, aq^\ell}^{-1})) \right)$ respects the following characterization (see \cite[Theorem 4.2]{hernandez_algebra_2011})
\begin{itemize}
\item[]it is an infinite sum of elements in $\ZZ[Y_{\overline{r}, a}(1 + A_{\overline{r}, aq}^{-1}), Y_{\overline{r'}, a}^{\pm 1}]_{a \in \CC^{\ast}, \overline{r} \neq \overline{r'}}$ for each $\overline{r} \in I_n$.
\end{itemize}
Actually, it suffices to check the analogue property at the level of the $\U_q(sl_\infty)$-crystal $\mathcal{M}(Y_{\ell, a}Y_{0, aq^\ell}^{-1})$: this holds by the description of it given above (see Remark \ref{remlienmoncriinf}).
\end{proof}

Let us give one of the main results of the paper.

\begin{thm}\label{thmqchaactreptorinf}
Assume that $n \geq 2$ is such that
\begin{align}\label{condexistmodinf}
(n \text{ is even and } \ell \leq n+1) \text{ or } \left(n \text{ is odd and } \ell \leq \frac{n+1}{2} \right).
\end{align}
Then $\phi_n \left( \chi_q(V(Y_{\ell, a}Y_{0, aq^\ell}^{-1})) \right)$ is the $q$--character of an actual representation of $\U_q(sl_{n+1}^{tor})$. We denote it $V_n(Y_{\ell, a}Y_{0, aq^\ell}^{-1})$.
\end{thm}

This result confirms Conjecture \ref{conjherninf} in the context of extremal fundamental loop weight modules.

\begin{proof}
This is a consequence of \cite[Proposition 4.4]{mansuy_extremal_2013}. In fact if (\ref{condexistmodinf}) holds, we have shown that there exists a sub-$\U_q(sl_{n+1}^{tor})$-module\footnote{It is denoted $\tilde{V} \left(\begin{array}{c} \ffbox{1} \\ \vdots \\ \ffbox{k} \end{array}_{aq^{\ell-1}} \right)$ in \cite{mansuy_extremal_2013}.} $V_n(Y_{\ell, a}Y_{0, aq^\ell}^{-1})$ of the fusion product of specialized vector representations with basis labelled by the set $\mathcal{T}_{[1, \ell]}$ of semi-standard tableaux of shape $(\ell)$. Furthermore for all $T \in \mathcal{T}_{[1, \ell]}$, $T_a$ is an $\ell$-weight vector of $\ell$-weight $\phi_n(m_T)$. The result follows directly.
\end{proof}

The main motivation of this study is the construction of extremal loop weight modules for $\U_q(sl_{n+1}^{tor})$. Actually we have

\pagebreak

\begin{thm}\label{thmelwmtorinf}\cite{mansuy_extremal_2013}
\begin{enumerate}
\item[(i)] Assume that $n \geq 2$. Then $V_n(Y_{1, a}Y_{0, aq}^{-1})$ is an extremal loop weight module of $\ell$-weight $Y_{1, a}Y_{0, aq}^{-1}$.
\item[(ii)] Assume that $n = 2r+1$ is odd $(r \geq 1)$. Then $V_n(Y_{r+1, a}Y_{0, aq^{r+1}}^{-1})$ admits an irreducible quotient which is an extremal loop weight module of $\ell$-weight $Y_{r+1, a}Y_{0, aq^{r+1}}^{-1}$.
\end{enumerate}
\end{thm}

\begin{rem}
In the other cases, the representations $V_n(Y_{\ell, a}Y_{0, aq^\ell}^{-1})$ are not $\ell$-extremal (see \cite{mansuy_extremal_2013}).
\end{rem}

\bibliographystyle{acm}

\end{document}